\documentclass{siamart251216}
\usepackage[T1]{fontenc}
\usepackage{lmodern,amssymb,mathtools,bm,xcolor}
\usepackage{booktabs,longtable,array,graphicx,microtype,enumitem}
\usepackage{tikz-cd}
\usepackage{placeins}
\graphicspath{{figures/}{paper/figures/}}
\usepackage{cite}
\usepackage{stmaryrd}

\allowdisplaybreaks
\setlist{nosep}
\newsiamremark{example}{Example}
\newcommand{\dd}{\mathrm d}
\newcommand{\id}{\operatorname{id}}
\newcommand{\Alt}{\operatorname{Alt}}
\newcommand{\ran}{\operatorname{ran}}
\newcommand{\spn}{\operatorname{span}}
\newcommand{\tr}{\operatorname{tr}}
\newcommand{\sym}{\operatorname{sym}}
\newcommand{\skw}{\operatorname{skw}}
\newcommand{\dev}{\operatorname{dev}}
\newcommand{\grad}{\operatorname{grad}}
\newcommand{\curl}{\operatorname{curl}}
\newcommand{\Div}{\operatorname{div}}
\newcommand{\hess}{\operatorname{hess}}
\newcommand{\defop}{\operatorname{def}}
\newcommand{\inc}{\operatorname{inc}}
\newcommand{\mskw}{\operatorname{mskw}}
\newcommand{\vskw}{\operatorname{vskw}}
\newcommand{\norm}[2][]{\left\lVert#2\right\rVert_{#1}}
\headers{Sparse-grid BGG sequences}{Yihanqi Hu and Yuwen Li}
\title{Sparse-grid discretizations of BGG sequences}
\author{Yihanqi Hu\thanks{School of Mathematical Sciences, Zhejiang University,
866 Yuhangtang Road, Hangzhou 310058, Zhejiang, People's Republic of China
(\email{huyihanqi@zju.edu.cn}, \email{liyuwen@zju.edu.cn}).}
\and Yuwen Li\footnotemark[1]}
\date{}
\begin{document}
\maketitle
\begin{abstract}
For Bernstein--Gelfand--Gelfand (BGG) complexes,
we develop the first sparse-grid finite element discretizations, for dramatically reducing the number of degrees of freedom used in tensor-product finite elements. Alpert wavelets and their
first and second primitives yield compatible hierarchies of discontinuous,
continuous, and continuously differentiable piecewise polynomials, that serve as one-dimensional ingredients of BGG sparse-grid finite elements. For the three-dimensional Hessian, elasticity, and div-div complexes, we construct a bounded commuting sparse-grid projection that yields stable discrete potentials, exactness and uniform inf-sup stability of associated mixed methods. Under mixed Sobolev regularity, we also prove sparse-grid approximation estimates with dimension-independent algebraic convergence rates in the number of degrees of freedom.
Experiments for a three-dimensional incompatibility problem and a two-dimensional biharmonic problem on an annulus illustrate the efficiency of sparse grids relative to tensor-product finite elements.
\end{abstract}

\section{Introduction}\label{sec:introduction}
Differential complexes is a useful tool for handling stability and conservation
laws of many PDEs. Their discrete counterparts, such as finite element
exterior calculus (FEEC)~\cite{AFW2006,AFW2010},
provide a systematic way to construct stable mixed finite element methods. FEEC theory is well established for de Rham complexes, such as PDEs posed in H(curl) and H(div), with differential forms and bounded commuting projections playing
central roles. However, PDEs arising from elasticity and general relativity often involve tensor fields and higher order differential operators, in which case FEEC theory for the discretization of corresponding complexes, such as Hessian, elasticity, and div-div complexes, requires further development. 

The Bernstein--Gelfand--Gelfand (BGG) construction connects such tensor-valued
complexes to diagrams of de Rham complexes. In~\cite{AH}, Arnold and Hu developed
an analytic framework for deriving cohomology, regular decomposition, and Poincar\'e inequalities of some complexes of tensor fields
from simpler ones; the Hessian, elasticity, and div-div complexes are key
examples. These constructions are built upon two operations:
differentiation along the input complexes and algebraic maps linking their
coefficient spaces. A BGG-compatible finite element discretization should be compatible with the linking maps and the algebraic projections used in BGG reduction.
On simplicial meshes, Chen and Huang~\cite{CHComplex2026} constructed three-dimensional
Hessian, elasticity, and div--div finite element complexes by applying BGG reduction to suitably restricted
smooth finite element de Rham complexes.

Tensor-product finite elements offer a particularly explicit setting for BGG discretization, which is often presented in a framework of double forms.
Double forms are differential forms whose coefficients are themselves
alternating tensors; see~\cite{HuLin2025,BKG2025} for recent systematic
developments. Their two form degrees organize the de Rham
complexes and the algebraic maps linking them in the BGG diagram.
Bonizzoni, Hu, Kanschat, and
Sap~\cite{BHKS} give a unified tensor-product construction using splines or
finite elements. Their one-dimensional spaces generate double-form spaces,
whose two form degrees determine the polynomial degree and continuity in each
coordinate. The linking maps preserve these coefficient spaces, and BGG
reduction yields conforming tensor-valued complexes. At minimal continuity,
with degree parameter $k\ge1$,
the finite element factors are $C^1$ piecewise polynomials of degree $k+2$,
$C^0$ piecewise polynomials of degree $k+1$, and discontinuous piecewise
polynomials of degree $k$. 

This tensor-product finite element construction (referred to as full-grid method) is the starting point
of the present work. For fixed
polynomial degree, a level-$n$ Cartesian discretization in dimension $d$ has
dimension of order $2^{nd}$. Sparse grids retain selected tensor products of
one-dimensional detail spaces and reduce the dimension to
$O(2^n(n+1)^{d-1})$. Moreover, the approximation power of sparse-grid finite elements is almost the same full-grid ones under sufficient mixed Sobolev regularity, see~\cite{BG2004}.
Hierarchical multiwavelet bases have also been used to construct sparse-grid
discontinuous Galerkin methods for high-dimensional PDEs~\cite{WTGC2016}.
However, a naive component-wise sparse-grid discretization of  vector-valued and tensor-valued problems cannot be stable or convergent. In such cases, the sparse-grid spaces must form a subcomplex
and admit stability bounds uniform in the sparse-grid level.

Sparse-grid discretization of de Rham complexes was initiated 
in~\cite{Gradinaru2002,GH2003}, where lowest-order FEEC on sparse grid was constructed.
In our previous work~\cite{HL}, we have designed higher-order sparse-grid de Rham complexes and constructed stable discrete potentials. To our knowledge, there is no sparse grid construction for BGG sequences in the literature. In the present work, we combined Alpert's discontinuous
$L^2$-orthogonal wavelets~\cite{Alpert1993} (Alpert detail) with their first and second primitives to obtain BGG-compatible finite element subcomplexes on sparse grids. We use a common downward-closed index set throughout the double-form
diagram. A coordinate-weight decomposition ensures that components
coupled by the linking maps share the same scalar sparse-grid space.
The linking maps and the associated algebraic projections and inverses
therefore preserve the sparse spaces, allowing BGG reduction to yield
sparse Hessian, elasticity, and div-div subcomplexes.

At the continuous level, existence of stable BGG potentials follows from regularized
Poincar\'e operators~\cite{CM2010,CapHu2023,AH}. In~\cite{HL}, we used a
tensor-product homotopy operator to construct stable discrete potentials of sparse-grid de Rham complexes. This homotopy could establish exactness of the input de Rham rows in the BGG diagram. Combining it with the BGG cohomology
argument~\cite[Theorem~6]{AH} yields exactness of our proposed sparse-grid BGG complexes. However, this argument does not provide stability of discrete potentials that are uniform in the refinement level. Therefore, we transfer the continuous BGG potentials into discrete ones on sparse grids, through a tensorized sparse-grid projection $Q_\mathcal{L}$ that commutes with both
differentiation and the algebraic linking maps. For each tensor block, we apply a first-order approximation estimate in a coordinate of maximal refinement and use $L^2$ stability in the remaining coordinates. The resulting projection $Q_\mathcal{L}$ is  $H^1\to L^2$ bounded uniformly with respect to the index set $\mathcal{L}$, and yields stable discrete potentials.
These potentials establish exactness and inf-sup stability of mixed methods.
We also prove mixed Sobolev approximation estimates for the sparse interpolants
and reduced tensor spaces. 

The rest of the paper is organized as follows. 
\Cref{sec:background} recalls the BGG framework, continuous regular potentials,
and full-grid spaces.
\Cref{sec:sparse} develops the compatible hierarchies and sparse BGG complexes.
Stable potentials and exactness are proved in
\cref{sec:potentials}, followed by approximation estimates in
\cref{sec:approximation}. \Cref{sec:experiments} then presents the numerical
experiments and specifies the boundary conditions and manufactured solutions
for the incompatibility problem on the cube and the biharmonic problem on an annulus.

\subsection{Notation}
Unless stated otherwise, let $\Omega=(0,1)^d$; set $I=(0,1)$ and $\mathbb{N}_0=\{0,1,2,\ldots\}$.
The polynomial degree parameter $k\ge1$ is a fixed integer.
Write $\bm x=(x_1,\ldots,x_d)$, let $\bm e_j$ be the $j$th coordinate vector,
and set $\bm0=(0,\ldots,0)$, $\bm1=(1,\ldots,1)$, and $\bm2=(2,\ldots,2)$.
For an interval $K$, $\mathbb{P}_j(K)$ is the univariate polynomial space of degree at
most $j$. Let $\mathbb{P}_j=\mathbb{P}_j(I)$ and set $\mathbb{P}_j=\{0\}$ for $j<0$.
For a multi-index $\bm{\alpha}$, define $|\bm{\alpha}|_1=\sum_j\alpha_j$ and
$\partial^{\bm\alpha}=\prod_j\partial_j^{\alpha_j}$.
All derivatives are distributional unless otherwise stated. 
The symbol $\lesssim$ allows constants depending on fixed $k,d$ but not on the
sparse-grid level; $A\eqsim B$ means $A\lesssim B\lesssim A$.
The same convention applies to generic constants $C$ unless additional
dependence is indicated.
The notation $(\cdot,\cdot)$ denotes the $L^2(\Omega)$ inner
product, and $P_M$ denotes orthogonal projection onto a closed subspace $M$.
For a differential operator $D$, its domain space is
\[
 H(D;\Omega)=\{u\in L^2(\Omega):Du\in L^2(\Omega)\},
\]
which is equipped with the graph norm $\norm[H(D)]u^2=\norm[L^2]u^2+\norm[L^2]{Du}^2$.

\section{BGG and full-grid discretization}\label{sec:background}
In this section, we recall the BGG construction and continuous regular-potential
results of Arnold and Hu~\cite{AH}, followed by the tensor-product finite
element discretizations of Bonizzoni et al.~\cite{BHKS}.

\subsection{Double forms and BGG complexes}\label{sec:continuous}
For integers $0\le a,b\le d$, let $\Alt^a\mathbb{R}^d$ be the space of alternating $a$-linear forms on $\mathbb{R}^d$, and set
$\Alt^{a,b}\mathbb{R}^d=\Alt^a\mathbb{R}^d\otimes\Alt^b\mathbb{R}^d$.
For $\bm{s}\in\mathcal B_a=\{\bm t\in\{0,1\}^d:|\bm t|_1=a\}$, we write
$e_{\bm s}=\dd x_{j_1}\wedge\cdots\wedge\dd x_{j_a}$, where $\{j_1,\ldots,j_a\}=\{j:s_j=1\}$ with $j_1<\cdots<j_a$ and $e_{\bm0}=1$.
These bases and their tensor products are orthonormal. A double-form field is
$u=\sum_{\bm s\in\mathcal B_a,\bm t\in\mathcal B_b}
u_{\bm s,\bm t}e_{\bm s}\otimes e_{\bm t}$, where each
$u_{\bm s,\bm t}$ is a scalar coefficient function. Define its exterior derivative
\begin{equation*}
 \dd u=\sum_{\bm s,\bm t,j}\partial_ju_{\bm s,\bm t}
 (\dd x_j\wedge e_{\bm s})\otimes e_{\bm t}.
\end{equation*}
For $a<d$ and $b>0$, define the algebraic linking map
$s^{a,b}:\Alt^{a,b}\mathbb{R}^d\to\Alt^{a+1,b-1}\mathbb{R}^d$ on coordinate basis forms by
\begin{equation*}
\begin{aligned}
 &s^{a,b}\bigl((\dd x_{i_1}\wedge\cdots\wedge\dd x_{i_a})
 \otimes(\dd x_{j_1}\wedge\cdots\wedge\dd x_{j_b})\bigr)\\
 &\quad=\sum_{r=1}^b(-1)^{r-1}
 (\dd x_{j_r}\wedge\dd x_{i_1}\wedge\cdots\wedge\dd x_{i_a})
 \otimes
 (\dd x_{j_1}\wedge\cdots\wedge\widehat{\dd x_{j_r}}
 \wedge\cdots\wedge\dd x_{j_b}),
\end{aligned}
\end{equation*}
with $i_1<\cdots<i_a$ and $j_1<\cdots<j_b$, and extend it linearly.
The hat denotes omission of that factor, and an empty wedge product equals $1$.
Thus each term moves one coordinate differential from the second form factor
to the front of the first. We set $s^{a,b}=0$ when $b=0$ or $a=d$.

We identify $L^2(\Omega;\Alt^{a,b}\mathbb{R}^d)$ with
$L^2(\Omega)\otimes\Alt^{a,b}\mathbb{R}^d$.
The algebraic linking map $s^{a,b}$ extends to double-form fields as
$S^{a,b}=\id_{L^2(\Omega)}\otimes s^{a,b}$, acting only on the
double-form factor and preserving scalar coefficients.
The following identities
\begin{equation*}
 \dd^2=0,\qquad \dd S+S\dd=0
\end{equation*}
hold distributionally.

\paragraph{Two-row reduction}
Consider two complexes $(A^\bullet,d_A)$ and
$(B^\bullet,d_B)$ linked by $S^i:B^i\to A^{i+1}$ with $d_AS=-Sd_B$ ($S^d=0$).
Their two-row BGG diagram is
\begin{equation*}
\begin{tikzcd}[column sep=2.0em,row sep=2.8em]
 0 \arrow[r] & A^0 \arrow[r,"d_A^0"]
 & A^1 \arrow[r,"d_A^1"] & \cdots \arrow[r,"d_A^{d-1}"]
 & A^d \arrow[r] & 0 \\
 0 \arrow[r] & B^0 \arrow[r,"d_B^0"] \arrow[ur,"S^0"]
 & B^1 \arrow[r,"d_B^1"] \arrow[ur,"S^1"]
 & \cdots \arrow[r,"d_B^{d-1}"] \arrow[ur,"S^{d-1}"]
 & B^d \arrow[r] & 0 .
\end{tikzcd}
\end{equation*}
For $J\in\{0,\ldots,d-1\}$, suppose $S^i$ is injective for $i\le J$, surjective for $i\ge J$, and
$S^J$ is bijective. The BGG reduction~\cite[Section 3]{AH} is
\begin{equation}\label{eq:BGG}
 V^i=\begin{cases}
 A^0,&i=0,\\
 A^i\cap\ran(S^{i-1})^\perp,&1\le i\le J,\\
 \ker S^i\subset B^i,&J<i\le d,
 \end{cases}
 \qquad
 D^i=\begin{cases}
 P_{\ran(S^i)^\perp}d_A^i,&i<J,\\
 d_B^J(S^J)^{-1}d_A^J,&i=J,\\
 d_B^i,&i>J.
 \end{cases}
\end{equation}
Orthogonal complements are taken in
the indicated ambient spaces.
The reduced spaces and operators form a complex, with the second-order
differential $D^J=d_B^J(S^J)^{-1}d_A^J$ taking values in $\ker S^{J+1}$.

\paragraph{Three-dimensional proxies}
In $d=3$ use cyclic bases
\[
 \begin{gathered}
 B_0=(1),\quad B_1=(\dd x,\dd y,\dd z),\\
 B_2=(\dd y\wedge\dd z,\dd z\wedge\dd x,\dd x\wedge\dd y),\quad
 B_3=(\dd x\wedge\dd y\wedge\dd z).
 \end{gathered}
\]
Let $e_i^{(a)}$ be the $i$th element of $B_a$, and identify
$\sum A_{ij}e_i^{(a)}\otimes e_j^{(b)}$ with its coefficient array $A$.
The first basis index $i$ labels rows and the second $j$ labels columns.
All three-component coefficient arrays, whether one row or one column,
are represented as column vectors.
The horizontal differential acts on the first form factor. Thus matrix
curl and divergence act columnwise, while the vector gradient is the
transpose of the Jacobian:
\begin{equation}\label{eq:column-ops}
 (\grad v)_{ij}=\partial_iv_j,\quad
 (\curl A)_{ij}=\epsilon_{i\ell r}\partial_\ell A_{rj},\quad
 (\Div A)_j=\partial_iA_{ij},
\end{equation}
with repeated indices summed and $\epsilon$ the Levi--Civita symbol,
oriented by $\epsilon_{123}=1$.

Let $I_3$ be the $3\times3$ identity matrix and write
$\mathbb{S}=\{A\in\mathbb{R}^{3\times3}:A^{\top}=A\}$, $\mathbb{T}=\{A\in\mathbb{R}^{3\times3}:\tr A=0\}$,
$\sym A=(A+A^{\top})/2$, $\skw A=(A-A^{\top})/2$, and
$\dev A=A-\tr(A)I_3/3$. Define
$\Xi A=A^{\top}-\tr(A)I_3$ and $\iota(c)=cI_3$.

The map $\mskw$ identifies $\mathbb{R}^3$ with the space of skew-symmetric
$3\times3$ matrices. For $v=(v_1,v_2,v_3)^{\top}$ and $A^{\top}=-A$, it and its inverse
are given by
\[
 \mskw(v)=
 \begin{pmatrix}
  0&-v_3&v_2\\
  v_3&0&-v_1\\
  -v_2&v_1&0
 \end{pmatrix},\qquad
 \mskw^{-1}(A)=
 \begin{pmatrix}A_{32}\\A_{13}\\A_{21}\end{pmatrix}.
\]
Set $\vskw=\mskw^{-1}\circ\skw$.

In these proxies, the bijective linking maps are
$S^{0,1}=\id$, $S^{1,2}=\Xi$, and $S^{2,3}=\id$.
For lower rows $b=1,2,3$, they occur at $J=b-1$.
Here $\Xi^{-1}A=A^{\top}-\tr(A)I_3/2$.

The reduced differential operators are
\begin{align}
 (\hess u)_{ij}&=\partial_i\partial_ju,&
 (\defop v)_{ij}&=\tfrac12(\partial_iv_j+\partial_jv_i),\notag\\
 \inc A&=\curl((\curl A)^{\top}),&
 \Div\Div A&=\sum_{i,j}\partial_i\partial_jA_{ij}.\label{eq:reduced-ops}
\end{align}
For symmetric $A$, $\tr(\curl A)=0$, so the general BGG differential
$\curl\Xi^{-1}\curl$ equals the displayed incompatibility operator.
The other first-order operators are the compositions indicated by their names.
 
\paragraph{Continuous complexes}
On $\Omega=(0,1)^3$, BGG reduction yields  Hessian, elasticity, and
div--div complexes, labeled $\mathfrak{H},\mathfrak{E},\mathfrak{D}$, respectively. Their initial kernels are
\begin{equation}\label{eq:initial-kernels}
 \mathcal K_{\mathfrak{H}}=\mathbb{P}_1(\Omega),\quad
 \mathcal K_{\mathfrak{E}}=\{a+b\times\bm x:a,b\in\mathbb{R}^3\},\quad
 \mathcal K_{\mathfrak{D}}=\{a+c\bm x:a\in\mathbb{R}^3,c\in\mathbb{R}\}.
\end{equation}
Here $\mathbb{P}_1(\Omega)$ denotes the space of affine scalar functions, and $H(D;\mathbb X)$ denotes
the previously defined maximal graph space with coefficient space $\mathbb X$
and no boundary conditions. A sequence is exact if the range of each map
equals the kernel of the next. The augmented continuous sequences are exact:
\[
\begin{aligned}
 0&\longrightarrow \mathcal K_{\mathfrak{H}}\hookrightarrow H(\hess;\mathbb{R})
 \xrightarrow{\hess}H(\curl;\mathbb{S})\xrightarrow{\curl}H(\Div;\mathbb{T})
 \xrightarrow{\Div}L^2(\Omega;\mathbb{R}^3)\longrightarrow0,\\[3pt]
 0&\longrightarrow \mathcal K_{\mathfrak{E}}\hookrightarrow H(\defop;\mathbb{R}^3)
 \xrightarrow{\defop}H(\inc;\mathbb{S})\xrightarrow{\inc}H(\Div;\mathbb{S})
 \xrightarrow{\Div}L^2(\Omega;\mathbb{R}^3)\longrightarrow0,\\[3pt]
 0&\longrightarrow \mathcal K_{\mathfrak{D}}\hookrightarrow H(\dev\grad;\mathbb{R}^3)
 \xrightarrow{\dev\grad}H(\sym\curl;\mathbb{T})\\
 &\xrightarrow{\sym\curl}H(\Div\Div;\mathbb{S})
 \xrightarrow{\Div\Div}L^2(\Omega)\longrightarrow0.
\end{aligned}
\]

We write $D_\star^i$, $\star\in\{\mathfrak{H},\mathfrak{E},\mathfrak{D}\}$,
$i=0,1,2$, for the successive differential operators
of each sequence.
The finite-dimensional coefficient spaces $(\mathbb X_\star^i)_{i=0}^3$
are $(\mathbb{R},\mathbb{S},\mathbb{T},\mathbb{R}^3)$, $(\mathbb{R}^3,\mathbb{S},\mathbb{S},\mathbb{R}^3)$, and
$(\mathbb{R}^3,\mathbb{T},\mathbb{S},\mathbb{R})$ for $\mathfrak{H},\mathfrak{E},\mathfrak{D}$, respectively.

Pauly and Zulehner develop regular potentials and decompositions for the
Hessian and div--div complexes~\cite{PZ2020} and the elasticity
complex~\cite{PZ2023}.
For $i=0,1,2$, let $f\in L^2(\Omega;\mathbb X_\star^{i+1})$, with
$D_\star^{i+1}f=0$ when $i<2$. There is a bounded linear map
$\mathcal R_\star^{i+1}$ on this data space, taking values in
$H^{\gamma_{\star,i}}(\Omega;\mathbb X_\star^i)$, such that
\begin{equation}\label{eq:regular-potential}
 D_\star^i\mathcal R_\star^{i+1}f=f,\qquad
 \norm[H^{\gamma_{\star,i}}]{\mathcal R_\star^{i+1}f}
 \le C_{\star,i}\norm[L^2]f,
\end{equation}
where the orders $(\gamma_{\star,0},\gamma_{\star,1},\gamma_{\star,2})$ are
$(2,1,1)$, $(1,2,1)$, and $(1,1,2)$ for $\mathfrak{H},\mathfrak{E},\mathfrak{D}$, respectively.
One linear choice maps each $f$ to its preimage of minimal
$H^{\gamma_{\star,i}}$ norm.

\subsection{Tensor-product finite element spaces}
Let $\mathcal T_m$ partition $I$ into $2^m$ equal intervals.
For $k\ge1$ set
\begin{equation}\label{eq:F}
 F_m^{[q]}=\{v\in C^{q-1}(I):v|_K\in\mathbb{P}_{k+q}(K),\ K\in\mathcal T_m\},
 \qquad q=0,1,2.
\end{equation}
Here $C^{-1}$ refers to finite element functions without any continuity constraint.
The spaces $F_m^{[0]}$, $F_m^{[1]}$, and $F_m^{[2]}$ consist of
discontinuous, continuous, and continuously differentiable piecewise
polynomials of degrees at most $k$, $k+1$, and $k+2$, respectively.

On a tensor mesh of level $m$, the full-grid finite element spaces of~\cite{BHKS} are
\begin{equation}\label{eq:full-background}
 V_{m,\mathrm{FG}}^{a,b}
 =\bigoplus_{\bm s\in\mathcal B_a,\bm t\in\mathcal B_b}
 \left(\bigotimes_{j=1}^d F_m^{[2-s_j-t_j]}(x_j)\right)
 e_{\bm s}\otimes e_{\bm t}.
\end{equation}

In $d=3$, use the proxy identifications of \cref{sec:continuous} and abbreviate
$V^{a,b}=V_{m,\mathrm{FG}}^{a,b}$. Let $\id$ on an arrow denote the identity
map. The four de Rham rows and their algebraic links take the form~\cite[Section 4.1]{AH}
\begin{equation*}
\begin{tikzcd}[column sep=2.6em,row sep=3.0em]
 0 \arrow[r] & V^{0,0} \arrow[r,"\grad"]
 & V^{1,0} \arrow[r,"\curl"] & V^{2,0} \arrow[r,"\Div"]
 & V^{3,0} \arrow[r] & 0 \\
 0 \arrow[r] & V^{0,1} \arrow[r,"\grad"] \arrow[ur,"\id"]
 & V^{1,1} \arrow[r,"\curl"] \arrow[ur,"2\vskw"]
 & V^{2,1} \arrow[r,"\Div"] \arrow[ur,"\tr"]
 & V^{3,1} \arrow[r] & 0 \\
 0 \arrow[r] & V^{0,2} \arrow[r,"\grad"] \arrow[ur,"-\mskw"]
 & V^{1,2} \arrow[r,"\curl"] \arrow[ur,"\Xi"]
 & V^{2,2} \arrow[r,"\Div"] \arrow[ur,"2\vskw"]
 & V^{3,2} \arrow[r] & 0 \\
 0 \arrow[r] & V^{0,3} \arrow[r,"\grad"] \arrow[ur,"\iota"]
 & V^{1,3} \arrow[r,"\curl"] \arrow[ur,"-\mskw"]
 & V^{2,3} \arrow[r,"\Div"] \arrow[ur,"\id"]
 & V^{3,3} \arrow[r] & 0 .
\end{tikzcd}
\end{equation*}

The algebraic diagonals are $S^{a,b}:V^{a,b}\to V^{a+1,b-1}$.
Applying the BGG reduction \eqref{eq:BGG} to adjacent rows yields full-grid
discretizations of the three continuous complexes in \cref{sec:continuous}.
For any compatible full or sparse diagram, we suppress the mesh subscript
and write the resulting reduced spaces $V_\star^i$, $i=0,1,2,3$, as
\begin{align}
 &V^{0,0}\xrightarrow{\hess}V^{1,1}\cap\mathbb{S}
 \xrightarrow{\curl}V^{2,1}\cap\mathbb{T}
 \xrightarrow{\Div}V^{3,1},\label{eq:Hcomplex}\\
 &V^{0,1}\xrightarrow{\defop}V^{1,1}\cap\mathbb{S}
 \xrightarrow{\inc}V^{2,2}\cap\mathbb{S}
 \xrightarrow{\Div}V^{3,2},\label{eq:Ecomplex}\\
 &V^{0,2}\xrightarrow{\dev\grad}V^{1,2}\cap\mathbb{T}
 \xrightarrow{\sym\curl}V^{2,2}\cap\mathbb{S}
 \xrightarrow{\Div\Div}V^{3,3}.\label{eq:Dcomplex}
\end{align}

The intersections impose pointwise algebraic constraints on the coefficient
arrays. The differentials $D_\star^i$ are the same operators as in
\cref{sec:continuous}.
For each fixed $\star$, let $(a_i,b_i)$ be the bidegree at position $i$,
$0\le i\le3$, in \eqref{eq:Hcomplex}--\eqref{eq:Dcomplex}.

For example, write $X=F_m^{[2]},Y=F_m^{[1]},Z=F_m^{[0]}$.
The componentwise tensor-product structure of the stress space $V^{2,2}$ is
\begin{equation*}
 \begin{pmatrix}
 X_x\otimes Z_y\otimes Z_z&Y_x\otimes Y_y\otimes Z_z&Y_x\otimes Z_y\otimes Y_z\\
 Y_x\otimes Y_y\otimes Z_z&Z_x\otimes X_y\otimes Z_z&Z_x\otimes Y_y\otimes Y_z\\
 Y_x\otimes Z_y\otimes Y_z&Z_x\otimes Y_y\otimes Y_z&Z_x\otimes Z_y\otimes X_z
 \end{pmatrix}.
\end{equation*}
The subscripts indicate the coordinate variables.
Double divergence in \eqref{eq:Dcomplex} differentiates diagonal entries in their $C^1$ direction
and off-diagonal entries in their two $C^0$ directions, ensuring conformity.

The tensor-product BGG construction in~\cite{BHKS} also covers splines.
Related spline de Rham work addresses commuting projections~\cite{BRSV2011},
nonuniform degrees and polar geometries~\cite{TH2021}. For simplicity of presentation, we focus
on the finite element setting; our sparse-grid construction can also be extended
to compatible spline hierarchies.

\section{Sparse-grid BGG complexes}\label{sec:sparse}
Sparse grids retain selected tensor-product detail spaces of \eqref{eq:full-background}.
To preserve the BGG diagram, differentiation must link the $C^1$, $C^0$, and
discontinuous details without increasing levels, while components coupled
by $S$ must share the same hierarchy and indices. 

\subsection{Compatible one-dimensional hierarchies}\label{sec:hierarchy}
For the spaces \eqref{eq:F}, integration gives
$\frac{\dd}{\dd x}F_m^{[q]}=F_m^{[q-1]}$ for $q=1,2$, and
\begin{equation*}
 \dim F_m^{[q]}=(k+1)2^m+q.
\end{equation*}
Let $P_m$ be the $L^2$ projection onto $F_m^{[0]}$ and set
$W_0^{[0]}=\mathbb{P}_k$, $W_m^{[0]}=F_m^{[0]}\cap(F_{m-1}^{[0]})^\perp$ for $m\ge1$.
Each parent interval $K=[a,b]\in\mathcal T_{m-1}$ is split at its midpoint.
Let $Z_K$ consist of piecewise $\mathbb{P}_k$ functions on its two children satisfying
\begin{equation}\label{eq:moments}
 \int_K zp=0,\qquad p\in\mathbb{P}_k(K).
\end{equation}
After zero extension to $I$, these local spaces give
$W_m^{[0]}=\bigoplus_{K\in\mathcal T_{m-1}}Z_K$.
Each $Z_K$ has dimension $k+1$. Choose an $L^2(K)$-orthonormal Alpert basis
$\{\psi_{m,K,r}\}_{r=0}^k$ by affine scaling of a common reference basis
on $[0,1]$, and extend it by zero outside $K$.
This is the discontinuous hierarchy of~\cite[Section 2.2]{HL}.
For $x\in K$, define the integration operator $J_K$ and its square by
\[
 J_Kz(x)=\int_a^xz(t)\,\dd t,\qquad
 (J_K\circ J_K)z(x)=J_K^2z(x)=\int_a^x(x-t)z(t)\,\dd t.
\]
Here powers denote composition and $J_K^0=\id$; extend the primitives by zero outside $K$.
For $m\ge1$, the three detail spaces are
\[
 W_m^{[q]}=\bigoplus_{K\in\mathcal T_{m-1}}
 \spn\{J_K^q\psi_{m,K,r}:0\le r\le k\},\qquad q=0,1,2.
\]
For $m\ge1$ and $z=\sum_Kz_K\in W_m^{[0]}$, $z_K\in Z_K$, the notation
$J^qz:=\sum_KJ_K^qz_K$ means integration separately on each parent,
so $W_m^{[q]}=J^qW_m^{[0]}$. 

At level zero, choose an $L^2(I)$-orthonormal basis
$\{\psi_{0,r}\}_{r=0}^k$ of $\mathbb{P}_k$ and let $J_I^q$ denote the same
primitive operators on $I$. We use the integrated bases
\[
 \begin{aligned}
 W_0^{[0]}&=\spn\{\psi_{0,r}:0\le r\le k\}=\mathbb{P}_k,\\
 W_0^{[1]}&=\spn\{1,J_I\psi_{0,r}:0\le r\le k\}=\mathbb{P}_{k+1},\\
 W_0^{[2]}&=\spn\{1,x,J_I^2\psi_{0,r}:0\le r\le k\}=\mathbb{P}_{k+2}.
 \end{aligned}
\]

The following lemma establishes local support of the integrated details
and compatibility of the hierarchies with differentiation.
\begin{lemma}\label{lem:hierarchy}
For $k\ge1$, $m\ge1$, $K=[a,b]\in\mathcal T_{m-1}$, and $z\in Z_K$, we have $J_K^qz\in F_m^{[q]}$ and $(J_K^qz)^{(r)}(a)=(J_K^qz)^{(r)}(b)=0$ for $q=1,2$, $r=0,\ldots,q-1$.
Moreover,
\begin{equation}\label{eq:hierarchy}
 F_M^{[q]}=\bigoplus_{m=0}^MW_m^{[q]},\qquad
 \frac{\dd}{\dd x}W_m^{[q]}=W_m^{[q-1]}\quad(q=1,2).
\end{equation}
For $m\ge1$,
$\dim W_m^{[q]}=(k+1)2^{m-1}$.
\end{lemma}
\begin{proof}
The cases $q=0,1$ are established in
\cite[Section 2.2]{HL}.
Fix $m\ge1$, $K=[a,b]\in\mathcal T_{m-1}$, and $z\in Z_K$.
For $q=2$, \eqref{eq:moments} with $p=1,t$ gives
\[
 J_Kz(b)=\int_Kz=0,\qquad
 J_K^2z(b)=\int_K(b-t)z(t)\,\dd t=0.
\]
Together with the zero left endpoint data and $(J_K^2z)'=J_Kz$, this shows
that the zero extension belongs to $F_m^{[2]}$.
If $g\in F_{m-1}^{[2]}\cap W_m^{[2]}$, then
\[
 g''\in F_{m-1}^{[0]}\cap W_m^{[0]}=\{0\}.
\]
Thus $g$ is affine, and its zero endpoint data give $g=0$.
Moreover, $(J^2z)''=z$ implies
\[
 \dim W_m^{[2]}=\dim W_m^{[0]}=(k+1)2^{m-1}
 =\dim F_m^{[2]}-\dim F_{m-1}^{[2]}.
\]
Hence $F_m^{[2]}=F_{m-1}^{[2]}\oplus W_m^{[2]}$, and induction proves
\eqref{eq:hierarchy}. 
\end{proof}
The derivative maps in \eqref{eq:hierarchy} are bijective for $m\ge1$.
For $m=0$, they are surjective and their kernels consist of constant functions.

Let $\{\widehat\psi_r\}_{r=0}^k$ be an $L^2(0,1)$-orthonormal basis
of $Z_{[0,1]}$, the reference detail space on the two halves of $[0,1]$.
On $K=[a,b]$, let $\xi=(x-a)/|K|$. For $x\in K$,
\begin{equation*}
 J_K^q\psi_{m,K,r}(x)
 =|K|^{q-1/2}(J_I^q\widehat\psi_r)(\xi),\qquad q=0,1,2.
\end{equation*}
The factor $|K|^{-1/2}$ preserves the $L^2$ normalization of the Alpert basis. 
The primitives are not $L^2$-orthonormal, while
the basis of $W_m^{[0]}$ is $L^2$-orthonormal.

\subsection{Canonical interpolation}
The degrees of freedom (dofs) in \cref{tab:canonical-dofs}, specified on each
cell $K=[a,b]$, define the one-dimensional interpolation operators
$\Pi_m^{[q]}:H^q(I)\to F_m^{[q]}$.
\begin{table}[htbp]
\centering
\caption{Dofs for the one-dimensional canonical interpolants $\Pi_m^{[q]}$.}
\label{tab:canonical-dofs}
\begin{tabular}{ccc}
\toprule
Space&Endpoint values&Interior moments\\\midrule
$\mathbb{P}_k$&none&$\int_Kvp$, $p\in\mathbb{P}_k$\\
$\mathbb{P}_{k+1}$&$v(a),v(b)$&$\int_Kvp$, $p\in\mathbb{P}_{k-1}$\\
$\mathbb{P}_{k+2}$&$v(a),v'(a),v(b),v'(b)$&$\int_Kvp$, $p\in\mathbb{P}_{k-2}$\\\bottomrule
\end{tabular}
\end{table}

\begin{proposition}\label{prop:interpolation}
For $q=0,1,2$, $\Pi_m^{[q]}$ are projections and satisfy
\begin{equation}\label{eq:Pi-commute}
 \frac{\dd}{\dd x}\Pi_m^{[q]}=\Pi_m^{[q-1]}\frac{\dd}{\dd x}\quad(q=1,2),\qquad
 \Pi_m^{[q]}\Pi_l^{[q]}=\Pi_{\min(m,l)}^{[q]}.
\end{equation}
\end{proposition}
\begin{proof}
The stated dofs are unisolvent, so the canonical interpolants
are projections, with $\Pi_m^{[0]}=P_m$.
For $q=1,2$, set $e=\Pi_m^{[q]}f-f$. On each cell $K=[a,b]$ and
for $p\in\mathbb{P}_{k+1-q}(K)$, the endpoint data and interior moments give
\[
 \int_K e'p=[ep]_a^b-\int_Kep'=0.
\]
For $q=2$, also $e'(a)=e'(b)=0$. Thus $(\Pi_m^{[q]}f)'$ has the
dofs of $\Pi_m^{[q-1]}f'$, proving commutation.
Iteration gives $(\Pi_m^{[q]}f)^{(q)}=P_mf^{(q)}$.
This derivative identity, together with $P_mP_l=P_{\min(m,l)}$ and
preservation of the initial derivatives of orders below $q$, proves the
composition identity in \eqref{eq:Pi-commute}.
\end{proof}

We set $\Pi_{-1}^{[q]}=0$ and define the one-dimensional differences
$\Delta_m^{[q]}:=\Pi_m^{[q]}-\Pi_{m-1}^{[q]}$, $m\in\mathbb{N}_0$, $q=0,1,2$.
Their ranges are $W_m^{[q]}$, and \eqref{eq:Pi-commute} gives
\begin{equation}\label{eq:Delta}
 \Delta_m^{[q]}\Delta_l^{[q]}=\delta_{ml}\Delta_m^{[q]},\qquad
 \frac{\dd}{\dd x}\Delta_m^{[q]}=\Delta_m^{[q-1]}\frac{\dd}{\dd x}.
\end{equation}
Here $\delta_{ml}$ is the Kronecker delta, and the last identity is for $q=1,2$.

These identities yield sparse commutation and reproduction
 in \cref{prop:cochain} and detail decay in \cref{sec:approximation}.

\subsection{Sparse-grid double-form spaces}\label{sec:diagram}
We use the double-form bases and operators fixed in \cref{sec:background}.
Let $\mathcal L\subset\mathbb{N}_0^d$ be a nonempty finite downward-closed set: if
$\bm m\in\mathcal L$ and $\bm0\le\bm l\le\bm m$, then $\bm l\in\mathcal L$.
For $\bm q\in\{0,1,2\}^d$ define
\begin{equation}\label{eq:G}
 F_{\mathcal L}^{\bm q}=\bigoplus_{\bm m\in\mathcal L}\bigotimes_{j=1}^dW_{m_j}^{[q_j]}(x_j),
 \qquad
 V_{\mathcal L}^{a,b}=\bigoplus_{\bm s\in\mathcal B_a,\bm t\in\mathcal B_b}
 F_{\mathcal L}^{\bm2-\bm s-\bm t}\,e_{\bm s}\otimes e_{\bm t}.
\end{equation}
Following the level numbering of~\cite{HL}, for $n\in\mathbb{N}_0$ use
\begin{equation*}
 \mathcal S_n=\{\bm m\in\mathbb{N}_0^d:|\bm m|_1\le n+d-1\},\qquad
 F_n^{\bm q}=F_{\mathcal S_n}^{\bm q},\qquad V_n^{a,b}=V_{\mathcal S_n}^{a,b}.
\end{equation*}
Write $\nu_n=n+d-1$ for the total-level cutoff; $n$ is the sparse-grid
level parameter, and $\nu_n$ is the maximal one-dimensional level.
The degree parameter $k$ is suppressed. All components use the same index set.

For comparison, let $\mathcal F_n=\{0,\ldots,n\}^d$. The 1D
decomposition \eqref{eq:hierarchy} gives
\begin{equation*}
 F_{n,\mathrm{FG}}^{\bm q}
 :=\bigotimes_{j=1}^dF_n^{[q_j]}(x_j)
 =F_{\mathcal F_n}^{\bm q}.
\end{equation*}
Thus \eqref{eq:G} includes the full spaces of~\cite{BHKS} by taking
$\mathcal L=\mathcal F_n$. Full grids retain $\max_jm_j\le n$, whereas sparse grids
retain $\sum_jm_j\le\nu_n$.

A basis double form $e_{\bm s}\otimes e_{\bm t}$ has coordinate-weight vector
$\bm c=\bm s+\bm t\in\{0,1,2\}^d$. Its $j$th entry counts the occurrences
of $\dd x_j$ across the two form factors. Define
\[
 E_{\bm c}^{a,b}=\spn\{e_{\bm s}\otimes e_{\bm t}:
 \bm s+\bm t=\bm c,\ |\bm s|_1=a,\ |\bm t|_1=b\}.
\]
A nonzero term of $s^{a,b}$ replaces $(s_j,t_j)=(0,1)$ by $(1,0)$,
so it preserves the coordinate-weight vector $\bm c$. Define
$s_{\bm c}:=s^{a,b}|_{E_{\bm c}^{a,b}}:
E_{\bm c}^{a,b}\to E_{\bm c}^{a+1,b-1}$.
The notation $\left.\cdot\right|_{\bm c}$ denotes restriction to the block
of the operator's source space with coordinate-weight vector $\bm c$.
The orthogonal block decomposition is
\begin{equation}\label{eq:weight-block}
 V_{\mathcal L}^{a,b}=\bigoplus_{\bm c}F_{\mathcal L}^{\bm2-\bm c}\otimes E_{\bm c}^{a,b},
 \qquad S|_{\bm c}=\id_{F_{\mathcal L}^{\bm2-\bm c}}\otimes s_{\bm c}.
\end{equation}

This block decomposition, together with the derivative relations in
\eqref{eq:hierarchy}, gives the compatibility needed for sparse-grid BGG reduction.
\begin{proposition}\label{prop:weights}
The spaces in \eqref{eq:G} form a BGG subdiagram:
\[
 \dd V_{\mathcal L}^{a,b}\subset V_{\mathcal L}^{a+1,b},\qquad
 S^{a,b}V_{\mathcal L}^{a,b}\subset V_{\mathcal L}^{a+1,b-1}.
\]
The discrete links inherit the injectivity and surjectivity of $s^{a,b}$.
Their inverses, when defined, and the BGG orthogonal projections preserve
each scalar detail space.
\end{proposition}
\begin{proof}
Set $\bm q=\bm2-\bm s-\bm t$. For a nonzero $j$th term of $\dd$,
\cref{lem:hierarchy} gives
\[
 s_j=0,\qquad q_j=2-t_j\ge1,\qquad
 \partial_jW_{m_j}^{[q_j]}\subset W_{m_j}^{[q_j-1]}.
\]
Consequently,
\[
\begin{aligned}
 \partial_j\left(\bigotimes_{\ell=1}^dW_{m_\ell}^{[q_\ell]}\right)
 &\subset\bigotimes_{\ell=1}^dW_{m_\ell}^{[q_\ell-\delta_{\ell j}]},\\
 q_\ell-\delta_{\ell j}&=2-(s_\ell+\delta_{\ell j})-t_\ell.
\end{aligned}
\]
These are weak derivatives by the $H^1$ regularity in $x_j$.
Since $\bm m$ is unchanged,
\[
 \dd V_{\mathcal L}^{a,b}\subset V_{\mathcal L}^{a+1,b}.
\]

Fix a coordinate-weight vector $\bm c$ and write
\[
 F_{\bm c}=F_{\mathcal L}^{\bm2-\bm c},\qquad
 S_{\bm c}=S|_{\bm c}.
\]
Equation \eqref{eq:weight-block} gives
$S^{a,b}V_{\mathcal L}^{a,b}\subset V_{\mathcal L}^{a+1,b-1}$ and
\[
 \ker S_{\bm c}=F_{\bm c}\otimes\ker s_{\bm c},\qquad
 \ran S_{\bm c}=F_{\bm c}\otimes\ran s_{\bm c}.
\]
Injectivity and surjectivity are therefore inherited on each block and,
by direct sums, on the full discrete spaces.

For bijective $s_{\bm c}$, we have $S_{\bm c}^{-1}=\id_{F_{\bm c}}\otimes s_{\bm c}^{-1}$.
The tensor-product inner product also gives
\[
\begin{aligned}
 P_{\ker S_{\bm c}}&=\id_{F_{\bm c}}\otimes P_{\ker s_{\bm c}},\\
 P_{(\ran S_{\bm c})^\perp}
 &=\id_{F_{\bm c}}\otimes P_{(\ran s_{\bm c})^\perp}.
\end{aligned}
\]
Complements are taken within the corresponding source or target blocks.
All these operators act only on the form factor, so each scalar detail
space and its level index remain unchanged.
\end{proof}

\subsection{Commuting sparse interpolation}
For $\bm q\in\mathbb{N}_0^d$, write
\[
 H_{\rm mix}^{\bm q}(\Omega)=\{u:\partial^{\bm\alpha}u\in L^2,
                    \ 0\le\bm\alpha\le\bm q\},\quad
 \norm[H_{\rm mix}^{\bm q}]u^2=\sum_{0\le\bm\alpha\le\bm q}
                              \norm[L^2]{\partial^{\bm\alpha}u}^2.
\]
This is the Hilbert tensor product of the one-dimensional $H^{q_j}(I)$ spaces with $H^0(I)=L^2(I)$.
Let $\Pi_{j,m}^{[q]}$ and $\Delta_{j,m}^{[q]}$ denote $\Pi_m^{[q]}$ and
$\Delta_m^{[q]}$, respectively, acting only in $x_j$.
Operators acting in different coordinates commute; their products denote composition.
On this space, define the tensor-product details and sparse interpolant by
\begin{equation*}
 \Delta_{\bm m}^{\bm q}:=\prod_{j=1}^d\Delta_{j,m_j}^{[q_j]},\qquad
 I_{\mathcal L}^{\bm q}=\sum_{\bm m\in\mathcal L}\Delta_{\bm m}^{\bm q}.
\end{equation*}
For double forms, define the mixed Sobolev space
\[
 \mathcal H^{a,b}_{\rm mix}=\bigoplus_{\bm s\in\mathcal B_a,\bm t\in\mathcal B_b}
 H_{\rm mix}^{\bm2-\bm s-\bm t}(\Omega)\,e_{\bm s}\otimes e_{\bm t}.
\]
It contains $V_{\mathcal L}^{a,b}$.
The operator $I_{\mathcal L}^{a,b}$ applies $I_{\mathcal L}^{\bm2-\bm s-\bm t}$
to each scalar coefficient $u_{\bm s,\bm t}$.
For $\mathcal L=\mathcal S_n$, abbreviate $I_n^{\bm q}=I_{\mathcal S_n}^{\bm q}$ and
$I_n^{a,b}=I_{\mathcal S_n}^{a,b}$.

\begin{proposition}\label{prop:cochain}
The operator $I_{\mathcal L}^{a,b}:\mathcal H^{a,b}_{\rm mix}\to V_{\mathcal L}^{a,b}$ is a projection and satisfies 
\begin{equation}\label{eq:commuting}
 \dd I_{\mathcal L}^{a,b}=I_{\mathcal L}^{a+1,b}\dd,\qquad
 S^{a,b}I_{\mathcal L}^{a,b}=I_{\mathcal L}^{a+1,b-1}S^{a,b}.
\end{equation}
It also commutes with the BGG algebraic projections and the inverses of bijective linking maps.
\end{proposition}
\begin{proof}
By \eqref{eq:Delta},
$\Delta_{\bm m}^{\bm q}\Delta_{\bm l}^{\bm q}
=(\prod_j\delta_{m_jl_j})\Delta_{\bm m}^{\bm q}$, hence
\[
\begin{aligned}
 (I_{\mathcal L}^{\bm q})^2
 &=\sum_{\bm m,\bm l\in\mathcal L}\Delta_{\bm m}^{\bm q}\Delta_{\bm l}^{\bm q}
 =I_{\mathcal L}^{\bm q},\\
 \ran I_{\mathcal L}^{\bm q}
 &=\bigoplus_{\bm m\in\mathcal L}\ran \Delta_{\bm m}^{\bm q}
 =F_{\mathcal L}^{\bm q}.
\end{aligned}
\]
This proves the projection property coefficientwise.
For the $j$th unit vector $\bm e_j$, \eqref{eq:Delta} gives
\[
 \partial_j I_{\mathcal L}^{\bm q}
 =\sum_{\bm m\in\mathcal L}\Delta_{\bm m}^{\bm q-\bm e_j}\partial_j
 =I_{\mathcal L}^{\bm q-\bm e_j}\partial_j
 \qquad(q_j\ge1).
\]
Each nonzero term of $\dd$ has $q_j=2-t_j\ge1$, so
$\dd I_{\mathcal L}^{a,b}=I_{\mathcal L}^{a+1,b}\dd$.
For the algebraic commutation, fix a coordinate-weight vector $\bm c$ and set
$\bm q=\bm2-\bm c$. By \eqref{eq:weight-block} and the proof of
\cref{prop:weights}, each relevant block map has the form
$\id\otimes t_{\bm c}$, where $t_{\bm c}:E\to F$ acts on the form spaces.
The source and target blocks use the same scalar interpolant, so
\[
 (\id\otimes t_{\bm c})(I_{\mathcal L}^{\bm q}\otimes\id_E)
 =I_{\mathcal L}^{\bm q}\otimes t_{\bm c}
 =(I_{\mathcal L}^{\bm q}\otimes\id_F)(\id\otimes t_{\bm c}).
\]
This proves commutation with $S$, the algebraic projections, and the
inverses of bijective links, blockwise and hence globally.
\end{proof}

\subsection{Reduced sparse complexes}\label{sec:reduced-sparse}
By \cref{prop:weights}, the kernels, range complements, and inverses used in
the BGG reduction \eqref{eq:BGG} preserve the sparse coefficient spaces and
the index set. Applying this reduction to adjacent rows of the sparse diagram
gives $V_{\star,\mathcal L}^i$ by inserting $V_{\mathcal L}^{a,b}$ in
\eqref{eq:Hcomplex}--\eqref{eq:Dcomplex}, with the same algebraic constraints.
For example, $V_{\mathfrak{E},\mathcal L}^1=V_{\mathcal L}^{1,1}\cap\mathbb{S}$ and
$V_{\mathfrak{E},\mathcal L}^2=V_{\mathcal L}^{2,2}\cap\mathbb{S}$.
Differentiation preserves the levels, and the coordinatewise continuity
ensures that the operators \eqref{eq:column-ops} and \eqref{eq:reduced-ops}
are the distributional differentials. Thus
\[
 V_{\star,\mathcal L}^i\subset H(D_\star^i;\Omega),\qquad
 D_\star^iV_{\star,\mathcal L}^i\subset V_{\star,\mathcal L}^{i+1},
 \qquad i=0,1,2,
\]
so the reduced sparse spaces form conforming subcomplexes.

Under the proxy identification, define
$\mathcal H_{{\rm mix},\star}^i=\mathcal H^{a_i,b_i}_{\rm mix}\cap L^2(\Omega;\mathbb X_\star^i)$.
The canonical interpolants preserve these constraints. Their restrictions are
\[
\begin{gathered}
 I_{\star,\mathcal L}^i:=I_{\mathcal L}^{a_i,b_i}|_{\mathcal H_{{\rm mix},\star}^i}:
 \mathcal H_{{\rm mix},\star}^i\to V_{\star,\mathcal L}^i,\\
 I_{\star,\mathcal L}^i|_{V_{\star,\mathcal L}^i}=\id,\qquad
 D_\star^iI_{\star,\mathcal L}^i=I_{\star,\mathcal L}^{i+1}D_\star^i.
\end{gathered}
\]
The commutation follows from \eqref{eq:BGG} and \eqref{eq:commuting} on these
mixed Sobolev domains. It includes the second-order differential
$D^J=d_B^J(S^J)^{-1}d_A^J$, since the interpolants commute with each factor.

For $\mathcal L=\mathcal S_n$, write $V_{\star,n}^i$ and $I_{\star,n}^i$.
For fixed $k,d$, the nonzero sparse and full-grid spaces have dimensions
\[
\begin{aligned}
 \dim F_n^{\bm q}\eqsim\dim V_{\star,n}^i
 &\eqsim2^n(n+1)^{d-1},\\
 \dim F_{n,\mathrm{FG}}^{\bm q}\eqsim\dim V_{\star,\mathcal F_n}^i
 &\eqsim2^{nd},
\end{aligned}
\]
with constants independent of $n$. The algebraic constraints in
\eqref{eq:weight-block} contribute only fixed multiplicities to these dimensions.
\section{Stable discrete potentials and exactness}\label{sec:potentials}
In this section, we consider the three reduced complexes on $\Omega=(0,1)^3$. Stable discrete potentials are constructed for the differential
operators $D_\star^0$, $D_\star^1$, and $D_\star^2$.
We construct commuting projections with uniform $H^1\to L^2$ bounds to
transfer the continuous regular potentials of \cref{sec:continuous} to the
sparse spaces.

\subsection{One-dimensional commuting projections}
We construct the projections from the $L^2$ orthogonal projection onto the
$C^1$ factor and the mean-zero primitive operator. This provides the
one-dimensional bounds needed to transfer regular potentials to sparse spaces.

\begin{lemma}\label{lem:smooth-1d}
For $k\ge1$, $m\ge0$, $q=0,1,2$,  there are projections
$Q_m^{[q]}:L^2(I)\to F_m^{[q]}$ satisfying
\begin{align}
 \norm[L^2]{Q_m^{[q]}f}&\le C_k\norm[L^2]f,\label{eq:Q-bound}\\
 (Q_m^{[q]}f)'&=Q_m^{[q-1]}f'
       \qquad(f\in H^1(I),\ q=1,2),\label{eq:Q-commute}\\
 \norm[L^2]{f-Q_m^{[q]}f}&\le C_k2^{-m}\norm[L^2]{f'}
       \qquad(f\in H^1(I)).\label{eq:Q-approx}
\end{align}
\end{lemma}
\begin{proof}
Set $h=2^{-m}$ and let $P_m^{[2]}$ be the $L^2$ orthogonal projection onto
$F_m^{[2]}$. Define the mean-zero primitive~\cite[Section 5.1]{HL} by
\[
 \mathcal J f(x)=\int_0^xf(t)\,\dd t
       -\int_0^1\int_0^yf(t)\,\dd t\,\dd y.
\]
Then $(\mathcal J f)'=f$, $\int_I\mathcal J f=0$, and Poincar\'e's inequality gives
$\norm[H^1]{\mathcal J f}\le C\norm[L^2]f$.
For $f\in H^1(I)$ and $\bar f=\int_I f$,
\[
 \mathcal J f'=f-\bar f,\qquad
 \mathcal J^2f'=\mathcal J f-\bar f\cdot(x-\tfrac12).
\]

\emph{Stability of the orthogonal projection.}
For $u\in H^2(I)$, let $J_mu$ be the piecewise cubic Hermite interpolant,
matching $u,u'$ at every node. Since $k\ge1$, $J_mu\in F_m^{[2]}$.
On each cell $K$, the error $e=u-J_mu$ has zero endpoint values and
derivatives. Integration by parts gives $\int_K e''p=0$ for
$p\in\mathbb{P}_1(K)$, so $(J_mu)''$ is the elementwise $\mathbb{P}_1$ orthogonal projection
of $u''$. Two elementwise Poincar\'e inequalities therefore give
\begin{equation}\label{eq:hermite-approx}
 \norm[L^2]{(J_mu)''}\le\norm[L^2]{u''},\qquad
 \norm[L^2]{u-J_mu}\le\frac{h^2}{\pi^2}\norm[L^2]{u''}.
\end{equation}
The polynomial inverse estimate
$\norm[L^2]{v_m''}\le C_kh^{-2}\norm[L^2]{v_m}$ for $v_m\in F_m^{[2]}$,
together with reproduction and $L^2$ contractivity, yields
\begin{equation}\label{eq:C1-stability}
 \begin{aligned}
 \norm[L^2]{(P_m^{[2]}u)''}
 &\le\norm[L^2]{(J_mu)''}
      +C_kh^{-2}\norm[L^2]{P_m^{[2]}(u-J_mu)}\\
 &\le C_k\norm[L^2]{u''}.
 \end{aligned}
\end{equation}
The standard equivalence
$\norm[H^2]u\eqsim\norm[L^2]u+\norm[L^2]{u''}$ gives uniform $H^2$ stability. The interpolation identity
$[L^2(I),H^2(I)]_{1/2}=H^1(I)$ then gives
\[
 \norm[H^1(I)]{P_m^{[2]}u}\le C_k\norm[H^1(I)]u,
 \qquad u\in H^1(I),\quad m\ge0.
\]

\emph{Construction and commutation.}
Set
\[
 Q_m^{[2]}=P_m^{[2]},\qquad
 Q_m^{[1]}=\frac{\dd}{\dd x}P_m^{[2]}\mathcal J,\qquad
 Q_m^{[0]}=\frac{\dd^2}{\dd x^2}P_m^{[2]}\mathcal J^2.
\]
Since $\mathcal J F_m^{[1]}\subset F_m^{[2]}$ and
$\mathcal J^2F_m^{[0]}\subset F_m^{[2]}$, these maps have the stated ranges
and reproduce them. The $H^1$ stability and \eqref{eq:C1-stability} give
\[
 \norm[L^2]{Q_m^{[1]}f}\le C_k\norm[H^1]{\mathcal J f}\le C_k\norm[L^2]f,
 \qquad
 \norm[L^2]{Q_m^{[0]}f}\le C_k\norm[L^2]{(\mathcal J^2f)''}
                         =C_k\norm[L^2]f.
\]
Together with contractivity for $q=2$, this proves \eqref{eq:Q-bound}.
Because $P_m^{[2]}$ reproduces constants and affine functions, the two
identities for $\mathcal J f'$ and $\mathcal J^2f'$ give \eqref{eq:Q-commute}.

\emph{First-order approximation.}
For $q=0,1$, the elementwise mean approximation and continuous piecewise linear
interpolant of $f$, respectively, belong to $F_m^{[q]}$.
The elementwise mean-zero and endpoint Poincar\'e inequalities give
\[
 \inf_{v_m\in F_m^{[q]}}\norm[L^2]{f-v_m}
 \le\frac{h}{\pi}\norm[L^2]{f'},\qquad q=0,1.
\]
Reproduction and \eqref{eq:Q-bound} imply
$\norm[L^2]{f-Q_m^{[q]}f}
\le(1+C_k)\inf_{v_m\in F_m^{[q]}}\norm[L^2]{f-v_m}$,
which proves \eqref{eq:Q-approx} for these two factors.
For $q=2$, let $T_m=\id-P_m^{[2]}$. Orthogonality and
\eqref{eq:hermite-approx} give
$\norm[L^2\to L^2]{T_m}\le1$ and
$\norm[H^2\to L^2]{T_m}\le h^2/\pi^2$.
Using $[L^2(I),H^2(I)]_{1/2}=H^1(I)$, operator interpolation with the
common target $L^2(I)$ yields
\[
 \norm[H^1\to L^2]{T_m}
 \le C\norm[L^2\to L^2]{T_m}^{1/2}
       \norm[H^2\to L^2]{T_m}^{1/2}
 \le Ch.
\]
Apply this bound to $f-\bar f$. Since $T_m$ annihilates constants,
Poincar\'e's inequality for mean-zero functions gives
$\norm[L^2]{T_mf}\le Ch\norm[L^2]{f'}$, completing \eqref{eq:Q-approx}.
\end{proof}

\subsection{Commuting sparse-grid projection}
Set $Q_{-1}^{[q]}=0$ and $\delta Q_m^{[q]}=Q_m^{[q]}-Q_{m-1}^{[q]}$.
Let $Q_{j,m}^{[q]}$ and $\delta Q_{j,m}^{[q]}$ denote $Q_m^{[q]}$ and
$\delta Q_m^{[q]}$, respectively, acting only in $x_j$. Define
\begin{equation}\label{eq:smooth-sparse}
 \mathcal Q_{\mathcal L}^{\bm q}
 =\sum_{\bm m\in\mathcal L}\prod_{j=1}^d\delta Q_{j,m_j}^{[q_j]}.
\end{equation}
Define the map
$\mathcal Q_{\mathcal L}^{a,b}:L^2(\Omega;\Alt^{a,b}\mathbb{R}^d)\to V_{\mathcal L}^{a,b}$
and its reduced restrictions by
\[
\begin{gathered}
 \mathcal Q_{\mathcal L}^{a,b}u:=
 \sum_{\bm s\in\mathcal B_a,\bm t\in\mathcal B_b}
 \bigl(\mathcal Q_{\mathcal L}^{\bm2-\bm s-\bm t}u_{\bm s,\bm t}\bigr)
 e_{\bm s}\otimes e_{\bm t},\\
 \mathcal Q_{\star,\mathcal L}^i:=
 \mathcal Q_{\mathcal L}^{a_i,b_i}|_{L^2(\Omega;\mathbb X_\star^i)}:
 L^2(\Omega;\mathbb X_\star^i)\to V_{\star,\mathcal L}^i.
\end{gathered}
\]
\begin{lemma}[Commuting sparse-grid projection]\label{lem:sparse-smooth}
The operator \eqref{eq:smooth-sparse} is a projection onto $F_{\mathcal L}^{\bm q}$ satisfying
\begin{equation}\label{eq:sparse-H1}
 \norm[L^2]{\mathcal Q_{\mathcal L}^{\bm q}v}\le C_{k,d}\norm[H^1(\Omega)]v,
\end{equation}
independently of $\mathcal L$. Its reduced projections satisfy
\[
 D_\star^i\mathcal Q_{\star,\mathcal L}^iu
 =\mathcal Q_{\star,\mathcal L}^{i+1}D_\star^iu,
 \qquad u\in H(D_\star^i;\Omega).
\]
\end{lemma}
\begin{proof}
Downward closure and \eqref{eq:hierarchy} give
\[
 F_{\mathcal L}^{\bm q}
 =\sum_{\bm\ell\in\mathcal L}\bigotimes_jF_{\ell_j}^{[q_j]}(x_j).
\]
Each summand in \eqref{eq:smooth-sparse} has range in the corresponding full
tensor-product space and hence in this sum. If $v_j\in F_{\ell_j}^{[q_j]}$, then
$\delta Q_m^{[q_j]}v_j=0$ for $m>\ell_j$, since both projections reproduce
$v_j$. For $v(\bm x)=\prod_jv_j(x_j)$, all surviving indices satisfy $\bm m\le\bm\ell$
and belong to $\mathcal L$. Telescoping gives
$\left(\prod_{j=1}^dQ_{j,\ell_j}^{[q_j]}\right)v=v$.
Linearity proves reproduction of the entire space.

For $m\ge1$, \eqref{eq:Q-bound}--\eqref{eq:Q-approx} give
\[
 \norm[L^2]{\delta Q_m^{[q]}w}\le3C_k2^{-m}\norm[L^2]{w'},
 \qquad \norm[L^2\to L^2]{\delta Q_m^{[q]}}\le2C_k.
\]
The zero-level increments are also $L^2$ bounded. These estimates extend to
Hilbert-valued functions by orthonormal expansion and summation of squares.
For $\bm m\ne\bm0$, choose $j_*$ with $m_{j_*}=M=\max_jm_j$.
The derivative estimate in $x_{j_*}$ and the $L^2$ bounds in the other
coordinates, together with Fubini's theorem, give
\[
 \left\|\left(\prod_{j=1}^d\delta Q_{j,m_j}^{[q_j]}\right)v\right\|_{L^2}
 \le C_{k,d}2^{-M}\norm[L^2]{\partial_{x_{j_*}}v}
 \le C_{k,d}2^{-M}\norm[H^1]v.
\]
The final bound also holds for the zero block, with $M=0$.
Since $\max_jm_j\ge|\bm m|_1/d$,
\[
 \sum_{\bm m\in\mathcal L}2^{-\max_jm_j}
 \le\sum_{\bm m\in\mathbb{N}_0^d}2^{-|\bm m|_1/d}
 =\prod_{j=1}^d\sum_{r=0}^{\infty}2^{-r/d}
 =(1-2^{-1/d})^{-d}.
\]
Summing the block bounds by the triangle inequality proves
\eqref{eq:sparse-H1}, with a constant independent of $\mathcal L$.

By \eqref{eq:Q-commute},
$\partial_jQ_{j,m}^{[q]}=Q_{j,m}^{[q-1]}\partial_j$ for $q=1,2$;
the same scalar operator
on each coordinate-weight block also commutes with the algebraic maps and their inverses
and projections. Hence the reduced commutation identities hold on smooth fields.
For fixed finite $\mathcal L$, the projections are $L^2$ bounded.
By density of smooth $\mathbb X_\star^i$-valued fields in the graph norm
and closedness of $D_\star^i$, these identities extend to
$H(D_\star^i;\Omega)$.
\end{proof}

\subsection{Level-uniform discrete potentials and exactness}
Let $Z_{\star,\mathcal L}^i=\ker(D_\star^i|_{V_{\star,\mathcal L}^i})$ for $i<3$,
and set $Z_{\star,\mathcal L}^3=V_{\star,\mathcal L}^3$.

\begin{theorem}[Level-uniform discrete potentials]\label{thm:all-potentials}
For every $k\ge1$, $\star\in\{\mathfrak{H},\mathfrak{E},\mathfrak{D}\}$, and $i=0,1,2$, there is a linear map
$\widehat R_{\star,\mathcal L}^{i+1}:Z_{\star,\mathcal L}^{i+1}\to V_{\star,\mathcal L}^{i}$
such that
\begin{equation*}
 D_\star^i\widehat R_{\star,\mathcal L}^{i+1}f=f,\qquad
 \norm[L^2]{\widehat R_{\star,\mathcal L}^{i+1}f}
 \le C_{\star,k}\norm[L^2]f.
\end{equation*}
The constants are independent of $\mathcal L$, in particular of the sparse level $n$.
\end{theorem}
\begin{proof}
By conformity, every discrete cocycle $f$ is a continuous cocycle. Continuous exactness
on the cube and \eqref{eq:regular-potential} give a regular preimage. Define
\[
 \widehat R_{\star,\mathcal L}^{i+1}f
 =\mathcal Q_{\star,\mathcal L}^{i}\mathcal R_\star^{i+1}f.
\]
Commutation and reproduction imply
$D_\star^i\widehat R_{\star,\mathcal L}^{i+1}f
=\mathcal Q_{\star,\mathcal L}^{i+1}f=f$.
Since every $\gamma_{\star,i}\ge1$, \eqref{eq:sparse-H1} and
\eqref{eq:regular-potential} give the stated bound. Both factors are linear.
\end{proof}

The exactness of sparse-grid BGG sequence is a corollary of Theorem \ref{thm:all-potentials}.
\begin{corollary}\label{thm:exact}
The following BGG sequence is exact
\begin{equation*}
 0\longrightarrow \mathcal K_\star\hookrightarrow V_{\star,\mathcal L}^0
 \xrightarrow{D_\star^0}V_{\star,\mathcal L}^1
 \xrightarrow{D_\star^1}V_{\star,\mathcal L}^2
 \xrightarrow{D_\star^2}V_{\star,\mathcal L}^3\longrightarrow0.
\end{equation*}
\end{corollary}

Taking $i=2$ in \cref{thm:all-potentials} gives
\begin{equation}\label{eq:inf-sup}
 \inf_{0\ne f_h\in V_{\star,\mathcal L}^3}
 \sup_{0\ne v_h\in V_{\star,\mathcal L}^2}
 \frac{(D_\star^2v_h,f_h)}
      {\norm[H(D_\star^2)]{v_h}\,\norm[L^2]{f_h}}
 \ge (1+C_{\star,k}^2)^{-1/2}
\end{equation}
by choosing $v_h=\widehat R_{\star,\mathcal L}^3f_h$ in \cref{thm:all-potentials}.

In the following examples, we identify the reduced spaces $V_{\star,\mathcal L}^i$
with their scalar, vector, or tensor proxies introduced in \cref{sec:continuous}.
We use the stable discrete potentials to verify the discrete
Babu\v{s}ka--Brezzi inf--sup condition and coercivity on the coupling kernel
for three mixed sparse-grid methods in $\mathbb{R}^3$.
Together with boundedness of the bilinear forms, these conditions establish
uniform stability and quasi-optimal error estimates.
In the error estimates, variables without the subscript $h$ denote the
corresponding continuous mixed solution.

\begin{example}[Mixed linear elasticity]\label{ex:elasticity}
Let $A$ be a bounded, symmetric, uniformly positive-definite compliance
tensor. For a body force $f\in L^2(\Omega;\mathbb{R}^3)$, consider the linear elasticity equations
$A\sigma=\defop u$, $-\Div\sigma=f$ in $\Omega$, with the homogeneous
displacement condition $u=0$ on $\partial\Omega$. The mixed method seeks
$(\sigma_h,u_h)\in V_{\mathfrak E,\mathcal L}^2\times V_{\mathfrak E,\mathcal L}^3$ such that
\begin{align*}
 (A\sigma_h,\tau_h)+(u_h,\Div\tau_h)&=0,
 &&\tau_h\in V_{\mathfrak E,\mathcal L}^2\subset H(\Div;\mathbb{S}),\\
 (\Div\sigma_h,v_h)&=-(f,v_h),&&v_h\in V_{\mathfrak E,\mathcal L}^3\subset L^2(\Omega;\mathbb{R}^3).
\end{align*}
The displacement condition is imposed weakly through the first equation.
The coupling satisfies \eqref{eq:inf-sup} with $\star=\mathfrak E$.
Coercivity on the divergence-free kernel follows from the uniform positive
definiteness of $A$. Hence the mixed sparse-grid method satisfies
\begin{align*}
 &\norm[H(\Div)]{\sigma-\sigma_h}+\norm[L^2]{u-u_h}\\
 &\quad\le C_{A,k}\left(
 \inf_{\tau_h\in V_{\mathfrak E,\mathcal L}^2}\norm[H(\Div)]{\sigma-\tau_h}
 +\inf_{v_h\in V_{\mathfrak E,\mathcal L}^3}\norm[L^2]{u-v_h}\right).
\end{align*}
\end{example}

\begin{example}[Mixed biharmonic equation]\label{ex:biharmonic}
Consider $\Delta^2u=f$ in $\Omega$ with the clamped boundary conditions
$u=\partial_{\boldsymbol n}u=0$ on $\partial\Omega$, where $\boldsymbol n$
is the outward unit normal. Introduce the moment tensor $\sigma=-\hess u$.
For $f\in L^2(\Omega)$, seek
$(\sigma_h,u_h)\in V_{\mathfrak D,\mathcal L}^2\times V_{\mathfrak D,\mathcal L}^3$ satisfying
\begin{equation}\label{eq:biharmonic-weak}
 \begin{aligned}
 (\sigma_h,\tau_h)+(u_h,\Div\Div\tau_h)&=0,
 &&\tau_h\in V_{\mathfrak D,\mathcal L}^2\subset H(\Div\Div;\mathbb{S}),\\
 (\Div\Div\sigma_h,v_h)&=-(f,v_h),
 &&v_h\in V_{\mathfrak D,\mathcal L}^3\subset L^2(\Omega).
 \end{aligned}
\end{equation}
The clamped conditions are imposed weakly through the first equation.
The coupling satisfies \eqref{eq:inf-sup} with $\star=\mathfrak D$.
The tensor mass bilinear form is coercive on the double-divergence-free
kernel. Consequently, the unique discrete solution satisfies
\begin{align*}
 &\norm[H(\Div\Div)]{\sigma-\sigma_h}+\norm[L^2]{u-u_h}\\
 &\quad\le C_k\left(
 \inf_{\tau_h\in V_{\mathfrak D,\mathcal L}^2}\norm[H(\Div\Div)]{\sigma-\tau_h}
 +\inf_{v_h\in V_{\mathfrak D,\mathcal L}^3}\norm[L^2]{u-v_h}\right).
\end{align*}
\end{example}

\begin{example}[incompatibility problem]\label{ex:inc}
The last example concerns $\inc(\inc U)=F$ in $\Omega$, with the homogeneous boundary conditions
\[
 U\cdot\bm n=\bm0\quad\text{on }\partial\Omega,\qquad
 \inc U\in H_0(\inc;\mathbb{S}).
\]
Here $H_0(\inc;\mathbb{S})$ denotes the closure of
$C_c^\infty(\Omega;\mathbb{S})$ in the $H(\inc)$ graph norm.
To remove nonuniqueness
caused by $\ker(\inc)=\defop H^1(\Omega;\mathbb{R}^3)$, we impose
$U\perp_{L^2}\ker(\inc)$ and introduce a Lagrange multiplier $p$
to enforce this constraint. Since only $\defop p$ enters the formulation,
we fix $p$ by requiring $(p,r)=0$ for all $r\in\mathcal K_{\mathfrak E}$.
Set $M_h=V_{\mathfrak E,\mathcal L}^0\cap\mathcal K_{\mathfrak E}^{\perp_{L^2}}$.
For $F\in L^2(\Omega;\mathbb{S})$, a mixed formulation for $\inc(\inc U)=F$ seeks $(U_h,p_h)\in V_{\mathfrak E,\mathcal L}^1\times M_h$ such that
\begin{equation}\label{eq:inc-weak}
 \begin{aligned}
 (\inc U_h,\inc V_h)+(\defop p_h,V_h)&=(F,V_h),&&V_h\in V_{\mathfrak E,\mathcal L}^1,\\
 (U_h,\defop q_h)&=0,&&q_h\in M_h.
 \end{aligned}
\end{equation}
\Cref{thm:all-potentials} implies that 
\begin{equation*}
 \inf_{0\ne q_h\in M_h}\sup_{0\ne V_h\in V_{\mathfrak E,\mathcal L}^1}
 \frac{(V_h,\defop q_h)}
      {\norm[H(\inc)]{V_h}\,\norm[H(\defop)]{q_h}}
 \ge c_k>0.
\end{equation*}
The verification of the inf--sup condition and kernel coercivity parallels
the classical mixed curl--curl analysis, with $\nabla$ and $\curl$ replaced
by $\defop$ and $\inc$, respectively, see \cite{Hiptmair2002}.
By \cref{thm:exact}, $Z_{\mathfrak E,\mathcal L}^1=\defop M_h$. For $U_h\in V_{\mathfrak E,\mathcal L}^1\cap(Z_{\mathfrak E,\mathcal L}^1)^{\perp_{L^2}}$,
the potential bound in \cref{thm:all-potentials} and orthogonality to the kernel
give the discrete Poincar\'e inequality
\begin{equation*}
 \norm[H(\inc)]{U_h}^2
 \le (1+C_{\mathfrak E,k}^2)\norm[L^2]{\inc U_h}^2.
\end{equation*}
Thus the energy bilinear form is coercive on this kernel, and the unique
discrete solution $(U_h,p_h)$ is quasi-optimal:
\begin{align*}
 &\norm[H(\inc)]{U-U_h}+\norm[H(\defop)]{p-p_h}\\
 &\quad\le C_k\left(
 \inf_{V_h\in V_{\mathfrak E,\mathcal L}^1}\norm[H(\inc)]{U-V_h}
 +\inf_{q_h\in M_h}\norm[H(\defop)]{p-q_h}\right).
\end{align*}
\end{example}

\section{Approximation properties}\label{sec:approximation}

For sparse-grid level $n$, we write $h_n=2^{-n}$.
We first treat scalar tensor-product spaces with integration orders
$\bm q\in\{0,1,2\}^d$. The estimates follow
from one-dimensional detail decay, anisotropic tail summation, and
orthogonality in the discontinuous directions.

\begin{lemma}\label{lem:decay}
If $q\in\{0,1,2\}$ and $\rho$ is an integer satisfying $q\le\rho\le k+1+q$, then
\begin{equation}\label{eq:decay}
 \norm[L^2]{\Delta_m^{[q]}u}\le C_k2^{-\rho m}\norm[H^\rho]u
 \qquad(u\in H^\rho(I),\ m\ge0).
\end{equation}
\end{lemma}
\begin{proof}
For $m\ge1$, let $K\in\mathcal T_{m-1}$, so $|K|=2^{-(m-1)}$, and set
$w=\Delta_m^{[q]}u$. By \eqref{eq:Delta},
$w^{(q)}=(P_m-P_{m-1})u^{(q)}$.
For $q=1,2$, the endpoint derivatives of $w$ of orders below $q$ vanish
by \cref{lem:hierarchy}. On $K$, $P_m-P_{m-1}$ is an orthogonal projector
annihilating $\mathbb{P}_k(K)$. Iterated Poincar\'e inequalities and
the order-$(\rho-q)$ Bramble--Hilbert estimate therefore give
\[
\begin{aligned}
 \norm[L^2(K)]w
 &\le C|K|^q\norm[L^2(K)]{w^{(q)}}\\
 &\le C|K|^q\inf_{p\in\mathbb{P}_k(K)}\norm[L^2(K)]{u^{(q)}-p}\\
 &\le C_k|K|^\rho\norm[L^2(K)]{u^{(\rho)}}.
\end{aligned}
\]
For $q=0$ the first inequality is an equality; for $\rho=q$ choose $p=0$.
Summing squares over parents proves \eqref{eq:decay} for $m\ge1$.
At $m=0$, use boundedness of the fixed interpolant on $H^q(I)$.
\end{proof}

Let $q_j\in\{0,1,2\}$ and let $\rho_j$ be positive integers satisfying
$q_j\le\rho_j\le k+1+q_j$.
Since $H_{\rm mix}^{\bm\rho}=\bigotimes_jH^{\rho_j}(I)$,
tensorizing \eqref{eq:decay} gives
\begin{equation}\label{eq:tensor-decay}
 \norm[L^2]{\Delta_{\bm m}^{\bm q}u}
 \le C_{k,d}2^{-\bm\rho\cdot\bm m}\norm[H_{\rm mix}^{\bm\rho}]u.
\end{equation}
For $q=1,2$,
$\norm[H^q]{f-\Pi_m^{[q]}f}\le C_q\norm[L^2]{(\id-P_m)f^{(q)}}\to0$
by the endpoint conditions and \eqref{eq:Pi-commute}, with uniform $H^q$
boundedness; $q=0$ uses orthogonal projection.
Telescoping in each coordinate therefore yields
\[
 \sum_{\bm0\le\bm m\le M\bm1}\Delta_{\bm m}^{\bm q}u
 =\left(\prod_{j=1}^d\Pi_{j,M}^{[q_j]}\right)u
 \longrightarrow u\quad\text{in }L^2(\Omega).
\]
Absolute summability in \eqref{eq:tensor-decay} allows regrouping by total
level, giving the sparse tail
\begin{equation}\label{eq:tail}
 u-I_n^{\bm q}u=\sum_{|\bm m|_1>\nu_n}\Delta_{\bm m}^{\bm q}u.
\end{equation}

For $A\subset\{1,\ldots,d\}$, write $A^c=\{1,\ldots,d\}\setminus A$,
$\bm m_A=(m_j)_{j\in A}$, $|\bm m_A|_1=\sum_{j\in A}m_j$,
$\mathbb{N}_0^A=\prod_{j\in A}\mathbb{N}_0$, and $I^A=\prod_{j\in A}I$.
For $\bm\ell\in\mathbb{N}_0^A$, the full index $(\bm\ell,\bm m_{A^c})$ has entry
$\ell_j$ on $A$ and $m_j$ on $A^c$; $L^2(I^A)$ uses variables $(x_j)_{j\in A}$.
Empty coordinate products contain one tuple of level zero;
$L^2(I^\emptyset)=\mathbb{R}$, and empty space tensor products equal $\mathbb{R}$.
An empty operator product is the identity on the space on which it acts.

\begin{lemma}[Anisotropic tail]\label{lem:tail}
If $\rho_j>0$, $\sigma=\min_j\rho_j$, and
$\mu=\#\{j:\rho_j=\sigma\}$, then for any integer total-level cutoff $M\ge0$,
\begin{equation*}
 \sum_{|\bm m|_1>M}2^{-\bm\rho\cdot\bm m}
 \le C_{\bm\rho,d}2^{-\sigma M}(M+1)^{\mu-1}.
\end{equation*}
\end{lemma}
\begin{proof}
Let $B=\{j:\rho_j>\sigma\}$ be the set of directions with larger decay exponents.
For $|\bm m|_1=s$, fix $\bm m_B\in\mathbb{N}_0^B$ with $|\bm m_B|_1\le s$.
The remaining $\mu$ levels sum to $s-|\bm m_B|_1$, giving
$\binom{s-|\bm m_B|_1+\mu-1}{\mu-1}\le C_\mu(s+1)^{\mu-1}$ choices.
Hence summing over $\bm m_B$ gives
\[
 \sum_{|\bm m|_1=s}2^{-\bm\rho\cdot\bm m}
 \le C_\mu2^{-\sigma s}(s+1)^{\mu-1}
       \prod_{j\in B}(1-2^{-(\rho_j-\sigma)})^{-1}.
\]
Summing these shells over $s>M$ by the weighted geometric tail estimate
of~\cite{HL} proves the claim.
\end{proof}

Applying the triangle inequality to \eqref{eq:tail} and using
\eqref{eq:tensor-decay} and \cref{lem:tail} with
$M=\nu_n=n+d-1$, we obtain
\begin{equation}\label{eq:basic-error}
 \norm[L^2]{u-I_n^{\bm q}u}
 \le C h_n^\sigma(n+1)^{\mu-1}\norm[H_{\rm mix}^{\bm\rho}]u.
\end{equation}
We next sharpen the logarithmic factor by using orthogonality only in the
discontinuous directions.

\begin{theorem}\label{thm:component}
Set $\rho_j=k+1+q_j$. If
$u\in H_{\rm mix}^{\bm\rho}$, then we have 
\begin{equation}\label{eq:component-error}
 \norm[L^2]{u-I_n^{\bm q}u}
 \le C\norm[H_{\rm mix}^{\bm\rho}]u
 \begin{cases}
 h_n^{k+3}(n+1)^{d-1},&q_j=2\text{ for all }j,\\
 h_n^{k+2}(n+1)^{w-1},&\begin{gathered}[t]q_j\ge1\text{ for all }j,\\
 w=\#\{j:q_j=1\}>0,\end{gathered}\\
 h_n^{k+1}(n+1)^{(v-1)/2},&v=\#\{j:q_j=0\}>0.
 \end{cases}
\end{equation}
\end{theorem}
\begin{proof}
The first two cases follow from \eqref{eq:basic-error}. For the third, use the partition
$A=\{j:q_j=0\}$, $|A|=v$, and $B=A^c=\{j:q_j>0\}$.
For each $\bm\ell\in\mathbb{N}_0^A$, group the tail by summing over $\bm m_B\in\mathbb{N}_0^B$:
\[
 F_{\bm\ell}:=\sum_{|\bm m_B|_1>\nu_n-|\bm\ell|_1}
               \Delta_{(\bm\ell,\bm m_B)}^{\bm q}u.
\]
Here $\Delta_{(\bm\ell,\bm m_B)}^{\bm q}$ uses levels $\ell_j$ on $A$ and $m_j$ on $B$.
Each $F_{\bm\ell}$ is a function on $\Omega$ belonging to
$(\bigotimes_{j\in A}W_{\ell_j}^{[0]})\otimes L^2(I^B)$.
Distinct groups are therefore orthogonal in a discontinuous coordinate,
and absolute convergence gives
\begin{equation}\label{eq:discontinuous-groups}
 \norm[L^2]{u-I_n^{\bm q}u}^2=\sum_{\bm\ell}\norm[L^2]{F_{\bm\ell}}^2.
\end{equation}
For $\tau(t)=\sum_{|\bm m_B|_1>t}2^{-\sum_{j\in B}\rho_jm_j}$,
the triangle inequality and \eqref{eq:tensor-decay} give
\[
 \norm[L^2]{F_{\bm\ell}}\le C2^{-(k+1)|\bm\ell|_1}
       \tau(\nu_n-|\bm\ell|_1)\norm[H_{\rm mix}^{\bm\rho}]u.
\]
For $b=|B|>0$, the rates are at least $k+2$, so \cref{lem:tail} gives
$\tau(t)\le C2^{-(k+2)t}(t+1)^{b-1}$ for $t\ge0$, while $\tau(t)\le C$
for $t<0$. There are $\binom{s+v-1}{v-1}\le C_v(s+1)^{v-1}$ indices
$\bm\ell$ with $|\bm\ell|_1=s$.
The part of \eqref{eq:discontinuous-groups} with $s>\nu_n$ is
bounded by $C2^{-2(k+1)\nu_n}(\nu_n+1)^{v-1}$ times the squared mixed norm.
For $s\le\nu_n$ the scalar factor is at most
\begin{align*}
 \sum_{s=0}^{\nu_n}(s+1)^{v-1}2^{-2(k+1)s}
 2^{-2(k+2)(\nu_n-s)}(\nu_n-s+1)^{2(b-1)}\\
 \le 2^{-2(k+1)\nu_n}(\nu_n+1)^{v-1}
       \sum_{t=0}^{\infty}2^{-2t}(t+1)^{2(b-1)}.
\end{align*}
If $B$ is empty, only $s>\nu_n$ contributes.
Taking the square root and using
$2^{-(k+1)\nu_n}(\nu_n+1)^{(v-1)/2}\le C_{k,d}h_n^{k+1}(n+1)^{(v-1)/2}$
proves the last case of \eqref{eq:component-error}.
\end{proof}

\subsection{Double forms and reduced graph norms}
For bidegree $(a,b)$ define the high-regularity coefficient space
\begin{equation*}
 \mathcal X_k^{a,b}=\bigoplus_{\bm s\in\mathcal B_a,\bm t\in\mathcal B_b}
 H_{\rm mix}^{\bm\rho(\bm s,\bm t)}e_{\bm s}\otimes e_{\bm t},\qquad
 \rho_j(\bm s,\bm t)=k+3-s_j-t_j,
\end{equation*}
with the norm given by the square root of the sum of squared mixed Sobolev
norms of the coefficients.

For each complex $\star$, let $(a_i,b_i)$ denote the bidegree at position $i$
in \eqref{eq:Hcomplex}--\eqref{eq:Dcomplex}. Define
$\mathcal X_{\star,k}^i=\mathcal X_k^{a_i,b_i}\cap L^2(\Omega;\mathbb X_\star^i)$,
with the norm inherited from $\mathcal X_k^{a_i,b_i}$.
The interpolant is defined on $\mathcal H_{{\rm mix},\star}^i$, while $\mathcal X_{\star,k}^i$
specifies the regularity used in the approximation estimates.

All algebraic maps that preserve the coordinate-weight vector are bounded
between the corresponding $\mathcal X_k^{a,b}$ spaces. Also
$\dd:\mathcal X_k^{a,b}\longrightarrow\mathcal X_k^{a+1,b}$ is bounded.
Indeed, differentiation in direction $j$ replaces $\bm\rho$ by
$\bm\rho-\bm e_j$, while algebraic maps combine identical mixed spaces.
The same argument applies to the reduced differentials. We present the sparse-grid interpolation error estimate in the following theorem.

\begin{theorem}\label{thm:interpolation-error}
For $u\in\mathcal X_k^{a,b}$, it holds that 
\begin{equation*}
 \begin{aligned}
 \norm[L^2]{u-I_n^{a,b}u}&\le C\mathcal E_{a,b}(n)\norm[\mathcal X_k^{a,b}]u,\\
 \mathcal E_{a,b}(n)&=\begin{cases}
 h_n^{k+3}(n+1)^{d-1},&a=b=0,\\
 h_n^{k+2}(n+1)^{a+b-1},&\min(a,b)=0<a+b,\\
 h_n^{k+1}(n+1)^{(\min(a,b)-1)/2},&\min(a,b)\ge1.
 \end{cases}
 \end{aligned}
\end{equation*}
For $d=3$, $i=0,1,2$, and $u\in\mathcal X_{\star,k}^i$,
\begin{equation*}
 \norm[H(D_\star^i)]{u-I_{\star,n}^iu}
 \le C\big(\mathcal E_{a_i,b_i}(n)\norm[\mathcal X_{\star,k}^i]u
 +\mathcal E_{a_{i+1},b_{i+1}}(n)
   \norm[\mathcal X_{\star,k}^{i+1}]{D_\star^iu}\big).
\end{equation*}
\end{theorem}
\begin{proof}
For the $L^2$ estimate, if $\min(a,b)=0$, count the $a+b$ singly occupied directions and use the
first two cases of \cref{thm:component}. Otherwise, each component with
$v=\#\{j:s_j=t_j=1\}>0$ satisfies the last case of that theorem, with
$v\le\min(a,b)$. Components with $v=0$ gain a factor $h_n$,
which absorbs any additional fixed power of $n+1$.
The orthonormal form basis gives
\begin{equation*}
\begin{aligned}
 \norm[L^2]{u-I_n^{a,b}u}^2
 &=\sum_{\bm s,\bm t}
   \norm[L^2]{u_{\bm s,\bm t}
        -I_n^{\bm2-\bm s-\bm t}u_{\bm s,\bm t}}^2\\
 &\le C\mathcal E_{a,b}(n)^2\norm[\mathcal X_k^{a,b}]u^2.
\end{aligned}
\end{equation*}

For the graph-norm estimate, set $e=u-I_{\star,n}^iu$.
The commuting property of \cref{prop:cochain} gives $D_\star^ie=D_\star^iu-I_{\star,n}^{i+1}D_\star^iu$.
Apply the $L^2$ error estimate to $e$ and $D_\star^ie$ yields the desired $H(D_\star^i)$-norm error estimate.
\end{proof}

The resulting error factors are summarized in \cref{tab:graph-error-factors}.
\begin{table}[htbp]
\centering
\caption{$H(D_\star^i)$ interpolation error for the three-dimensional BGG complexes.}
\label{tab:graph-error-factors}
\begin{tabular}{cccc}
\toprule
Complex&First operator&Second operator&Third operator\\\midrule
Hessian&$h_n^{k+1}$&$h_n^{k+1}$&$h_n^{k+1}$\\
elasticity&$h_n^{k+1}$&$h_n^{k+1}(n+1)^{1/2}$&$h_n^{k+1}(n+1)^{1/2}$\\
div-div&$h_n^{k+1}$&$h_n^{k+1}(n+1)^{1/2}$&$h_n^{k+1}(n+1)$\\\bottomrule
\end{tabular}
\end{table}
Let $\beta\ge0$ be the fixed exponent of $n+1$ in the relevant error bound.
Let $N_n$ and $N_n^{\rm FG}$ denote the sparse-grid and full-grid dimensions,
respectively. In terms of $N_n$, the factor $h_n^{k+1}(n+1)^\beta$ becomes
\begin{equation*}
 C N_n^{-(k+1)}(1+\log N_n)^{(k+1)(d-1)+\beta},\qquad
 N_n\eqsim2^n(n+1)^{d-1}.
\end{equation*}
By \cref{sec:reduced-sparse}, the full-grid rate under the same mixed regularity is
$h_n^{k+1}\eqsim(N_n^{\rm FG})^{-(k+1)/d}$. Under weaker mixed regularity, apply
\eqref{eq:basic-error} separately to $u$ and $D_\star^iu$, provided both
fields satisfy its regularity hypotheses.

\section{Numerical experiments}\label{sec:experiments}
In this section, we compare accuracy of sparse-grid and full-grid finite element discretizations. To describe sparse
and full methods together, we take $\mathcal L=\mathcal S_n$ for sparse grids or
$\mathcal L=\mathcal F_n$ for full grids.

\subsection{Incompatibility problem}\label{sec:inc}
On $\Omega=(0,1)^3$, we use the natural-boundary formulation \eqref{eq:inc-weak}
with the full spaces $V_{\mathfrak E,\mathcal L}^1\subset H(\inc;\mathbb{S})$ and
$M_h=V_{\mathfrak E,\mathcal L}^0\cap\mathcal K_{\mathfrak E}^{\perp_{L^2}}$,
where $V_{\mathfrak E,\mathcal L}^0\subset H^1(\Omega;\mathbb{R}^3)$.
Set $g(t)=\sin^6(\pi t)$ and choose
\[
 T=\operatorname{diag}\bigl(g(y)g(z),g(x)g(z),g(x)g(y)\bigr),
 \qquad U=\inc T,\qquad F=\inc(\inc U).
\]
These fields satisfy
the boundary conditions in Example~\ref{ex:inc} and
$U\perp_{L^2}\ker(\inc)$, so the exact multiplier is $p=0$.
We report the relative errors $e_U=\norm[L^2]{U-U_h}/\norm[L^2]{U}$,
 $e_{\inc}=\norm[L^2]{\inc(U-U_h)}/\norm[L^2]{\inc U}$
against $N=\dim V_{\mathfrak E,\mathcal L}^1+\dim M_h$.
\begin{figure}[!htbp]
 \centering
 \includegraphics[width=.48\linewidth]{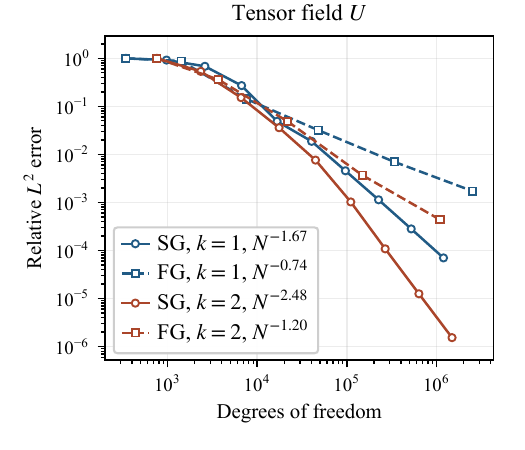}\hfill
 \includegraphics[width=.48\linewidth]{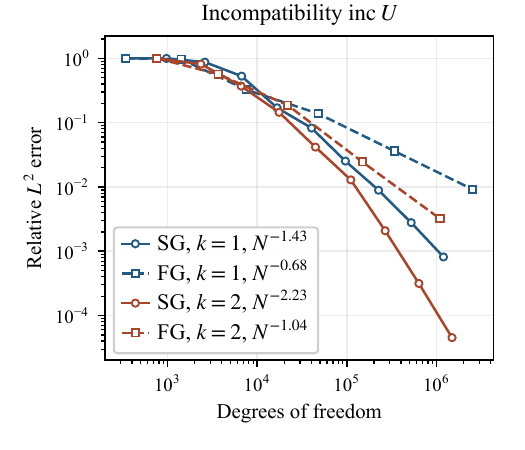}
 \caption{Incompatibility problem: relative $L^2$ errors of
 $U$ (left) and $\inc U$ (right) versus total dofs, for $k=1,2$. Legend exponents in $e\approx CN^{-p}$ are fitted to the last three points.}\label{fig:inc}
\end{figure}
On the finer levels, \cref{fig:inc} shows smaller errors and faster decay in $N$
for sparse grids.

\subsection{Mixed biharmonic problem on an annulus}\label{sec:biharmonic-example} Although our theory is developed on cubical domains, we test the performance
of the proposed sparse-grid BGG on the following annulus as a numerical extension:
\[
 \Omega=\{\bm x\in\mathbb R^2:\,0.05=r_{\rm in}<|\bm x|<r_{\rm out}=0.35\}.
\]
We solve $\Delta^2u=f$ with $u=\partial_{\boldsymbol n}u=0$ on $\partial\Omega$.
Set $L=r_{\rm out}-r_{\rm in}$, $\theta_p=\pi(p-1)$ and $\widehat\Omega=(0,1)^2$.
The standard reference-patch parametrizations~\cite{BRSV2011,BDKSVW2020} are
\[
\begin{aligned}
 \Phi_p(\xi,\eta)&=r(\cos\theta,\sin\theta),\qquad
 \Omega_p=\Phi_p(\widehat\Omega),\quad p=1,2,\\
 r(\xi)&=r_{\rm in}+L\xi,\qquad \theta=\theta_p+\pi\eta.
\end{aligned}
\]
We use the following spaces on the reference domain $\widehat{\Omega}$:
\[
\begin{gathered}
\widehat Q_{\mathcal L}=F_{\mathcal L}^{(0,0)},\\
 \widehat\Sigma_{\mathcal L}
 =\left\{\widehat\tau=\begin{pmatrix}\widehat A&\widehat B\\\widehat B&\widehat C\end{pmatrix}:
 (\widehat A,\widehat B,\widehat C)\in F_{\mathcal L}^{(2,0)}\times F_{\mathcal L}^{(1,1)}
                       \times F_{\mathcal L}^{(0,2)}\right\}.
\end{gathered}
\]
For $\widehat\tau\in\widehat\Sigma_{\mathcal L}$ and
$\widehat U\in\widehat Q_{\mathcal L}$, define the functions
$A,B,C,U$ of $(r,\theta)$ by
\[
 (A,B,C,U)(r,\theta)
 =(L^2\widehat A,L\pi\widehat B,\pi^2\widehat C,\widehat U)
 \!\left(\frac{r-r_{\rm in}}L,\frac{\theta-\theta_p}{\pi}\right).
\]

In the polar frame $\bm e_r=(\cos\theta,\sin\theta)^\top$,
$\bm e_\theta=(-\sin\theta,\cos\theta)^\top$, we write
$\tau_{ij}(r,\theta)=\bm e_i^\top\tau(r\bm e_r)\bm e_j$ and define
$(\mathcal I g)(r,\theta)=\int_{r_{\rm in}}^r g(s,\theta)\,ds$.
Define the moment transformation $\tau=\mathcal M_p\widehat\tau$ and scalar map by
\begin{equation*}
 \tau_{rr}=\frac{A+\mathcal I(rC)}{r},\quad
 \tau_{r\theta}=B-\frac{\mathcal I B}{r},\quad
 \tau_{\theta\theta}=rC,\quad
 (v\circ\Phi_p)(\xi,\eta)=\frac{\widehat U(\xi,\eta)}{r(\xi)}.
\end{equation*}
Differentiating the moment transformation gives
\[
(\Div\Div\tau)\circ\Phi_p=r(\xi)^{-1}\Div\Div\widehat\tau.
\]
On the interface $\Gamma$ at $\theta=\pi$ and $0$, the normal-normal
moment and effective shear are $rC$ and $2\partial_rB+\partial_\theta C$.
With $\llbracket\cdot\rrbracket$ denoting trace jumps, define
\[
\begin{aligned}
 \Sigma_{\mathcal L}
 &=\{\tau:\ \tau|_{\Omega_p}\in\mathcal M_p\widehat\Sigma_{\mathcal L},
      \ \llbracket B\rrbracket=\llbracket C\rrbracket=\llbracket\partial_\theta C\rrbracket=0\text{ on }\Gamma,\ p=1,2\},\\
 Q_{\mathcal L}
 &=\{v\in L^2(\Omega):\ r(\xi)(v\circ\Phi_p)\in\widehat Q_{\mathcal L},
                                     \ p=1,2\}.
\end{aligned}
\]
By construction, $\Sigma_\mathcal{L}\subset H({\rm divdiv},\Omega)$ and $\Div\Div\Sigma_{\mathcal L}\subset Q_{\mathcal L}$. The sparse-grid method
on the annulus is of the form \eqref{eq:biharmonic-weak} with $V_{\mathfrak D,\mathcal L}^2\times V_{\mathfrak D,\mathcal L}^3$ being replaced by $\Sigma_\mathcal{L}\times Q_\mathcal{L}$.

We set $\sigma=-\hess u$ and $u=\phi^2e^{x+y}$ with $\phi=\frac{4(|\bm x|^2-r_{\rm in}^2)(r_{\rm out}^2-|\bm x|^2)}
                 {(r_{\rm out}^2-r_{\rm in}^2)^2}$.
We plot $e_w=\norm[L^2(\Omega)]{w-w_h}/\norm[L^2(\Omega)]{w}$ with $w=\sigma$, $u$
against $N=\dim\Sigma_{\mathcal L}+\dim Q_{\mathcal L}$.
At comparable dof counts, \cref{fig:biharmonic} shows smaller sparse-grid
errors on the finer levels.

\begin{figure}[!ht]
 \centering
 \includegraphics[width=.48\linewidth]{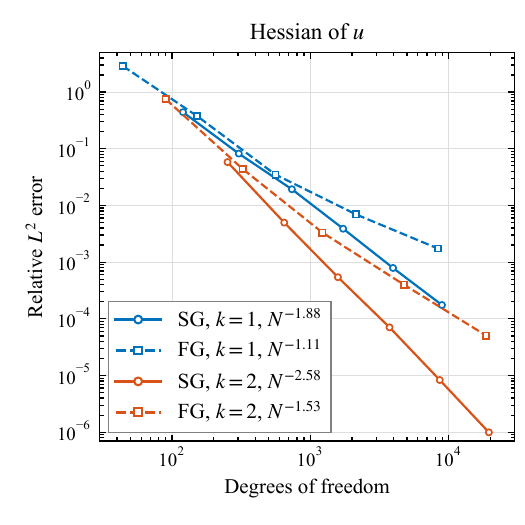}\hfill
 \includegraphics[width=.48\linewidth]{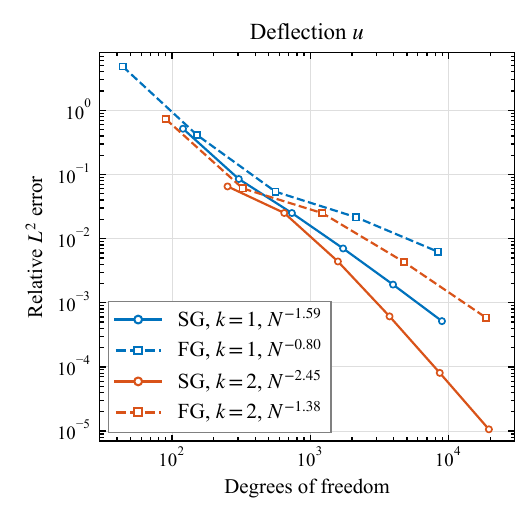}
 \caption{Mapped annulus with two patches: relative Hessian (left) and
 deflection (right) $L^2$ errors versus total dofs, $k=1,2$. Legend exponents in $e\approx CN^{-p}$
 are fitted to the last three points.}
 \label{fig:biharmonic}
\end{figure}
\FloatBarrier

\section*{Declaration of AI use}
During the preparation of this work, the authors used GPT-5.6 sol and GPT-6.0 Astra (OpenAI) to polish the writing,
generate numerical codes, and conduct computational experiments. The authors reviewed and
edited all generated output and take full responsibility for the content of the paper.

\bibliographystyle{siamplain}

\begin{thebibliography}{10}

\bibitem{Alpert1993}
{\sc B.~K. Alpert}, {\em A class of bases in {$L^2$} for the sparse representation of integral operators}, SIAM J. Math. Anal., 24 (1993), pp.~246--262, \url{https://doi.org/10.1137/0524016}.

\bibitem{AFW2006}
{\sc D.~N. Arnold, R.~S. Falk, and R.~Winther}, {\em Finite element exterior calculus, homological techniques, and applications}, Acta Numer., 15 (2006), pp.~1--155, \url{https://doi.org/10.1017/S0962492906210018}.

\bibitem{AFW2010}
{\sc D.~N. Arnold, R.~S. Falk, and R.~Winther}, {\em Finite element exterior calculus: from {Hodge} theory to numerical stability}, Bull. Amer. Math. Soc. (N.S.), 47 (2010), pp.~281--354, \url{https://doi.org/10.1090/S0273-0979-10-01278-4}.

\bibitem{AH}
{\sc D.~N. Arnold and K.~Hu}, {\em Complexes from complexes}, Found. Comput. Math., 21 (2021), pp.~1739--1774, \url{https://doi.org/10.1007/s10208-021-09498-9}.

\bibitem{BKG2025}
{\sc Y.~Berchenko-Kogan and E.~S. Gawlik}, {\em Finite element spaces of double forms}.
\newblock arXiv preprint, 2025, \url{https://arxiv.org/abs/2505.17243}.

\bibitem{BHKS}
{\sc F.~Bonizzoni, K.~Hu, G.~Kanschat, and D.~Sap}, {\em Discrete tensor product {BGG} sequences: splines and finite elements}, Math. Comp., 94 (2025), pp.~517--549, \url{https://doi.org/10.1090/mcom/3969}.

\bibitem{BDKSVW2020}
{\sc A.~Buffa, J.~D{\"o}lz, S.~Kurz, S.~Sch{\"o}ps, R.~V{\'a}zquez, and F.~Wolf}, {\em Multipatch approximation of the {de Rham} sequence and its traces in isogeometric analysis}, Numer. Math., 144 (2020), pp.~201--236, \url{https://doi.org/10.1007/s00211-019-01079-x}.

\bibitem{BRSV2011}
{\sc A.~Buffa, J.~Rivas, G.~Sangalli, and R.~V{\'a}zquez}, {\em Isogeometric discrete differential forms in three dimensions}, SIAM J. Numer. Anal., 49 (2011), pp.~818--844, \url{https://doi.org/10.1137/100786708}.

\bibitem{BG2004}
{\sc H.-J. Bungartz and M.~Griebel}, {\em Sparse grids}, Acta Numer., 13 (2004), pp.~147--269, \url{https://doi.org/10.1017/S0962492904000182}.

\bibitem{CapHu2023}
{\sc A.~{\v{C}ap} and K.~Hu}, {\em Bounded {Poincar{\'e}} operators for twisted and {BGG} complexes}, J. Math. Pures Appl., 179 (2023), pp.~253--276, \url{https://doi.org/10.1016/j.matpur.2023.09.008}.

\bibitem{CHComplex2026}
{\sc L.~Chen and X.~Huang}, {\em Complexes from complexes: finite element complexes in three dimensions}, Math. Comp., 95 (2026), pp.~1083--1142, \url{https://doi.org/10.1090/mcom/4079}.

\bibitem{CM2010}
{\sc M.~Costabel and A.~McIntosh}, {\em On {Bogovski{\u{\i}}} and regularized {Poincar{\'e}} integral operators for {de Rham} complexes on {Lipschitz} domains}, Math. Z., 265 (2010), pp.~297--320, \url{https://doi.org/10.1007/s00209-009-0517-8}.

\bibitem{GH2003}
{\sc V.~Gradinaru and R.~Hiptmair}, {\em Mixed finite elements on sparse grids}, Numer. Math., 93 (2003), pp.~471--495, \url{https://doi.org/10.1007/s002110100382}.

\bibitem{Gradinaru2002}
{\sc V.~C. Gr{\u{a}}dinaru}, {\em {Whitney} elements on sparse grids}, PhD thesis, Universit{\"a}t T{\"u}bingen, T{\"u}bingen, Germany, 2002, \url{https://publikationen.uni-tuebingen.de/xmlui/handle/10900/48372}.

\bibitem{Hiptmair2002}
{\sc R.~Hiptmair}, {\em Finite elements in computational electromagnetism}, Acta Numer., 11 (2002), pp.~237--339, \url{https://doi.org/10.1017/S0962492902000041}.

\bibitem{HuLin2025}
{\sc K.~Hu and T.~Lin}, {\em Finite element form-valued forms: construction}.
\newblock arXiv preprint, 2025, \url{https://arxiv.org/abs/2503.03243}.

\bibitem{HL}
{\sc Y.~Hu and Y.~Li}, {\em Higher-order finite element {de Rham} complexes on sparse grids}.
\newblock arXiv preprint, 2026, \url{https://arxiv.org/abs/2609.27768}.

\bibitem{PZ2020}
{\sc D.~Pauly and W.~Zulehner}, {\em The {divDiv}-complex and applications to biharmonic equations}, Appl. Anal., 99 (2020), pp.~1579--1630, \url{https://doi.org/10.1080/00036811.2018.1542685}.

\bibitem{PZ2023}
{\sc D.~Pauly and W.~Zulehner}, {\em The elasticity complex: compact embeddings and regular decompositions}, Appl. Anal., 102 (2023), pp.~4393--4421, \url{https://doi.org/10.1080/00036811.2022.2117497}.

\bibitem{TH2021}
{\sc D.~Toshniwal and T.~J.~R. Hughes}, {\em Isogeometric discrete differential forms: non-uniform degrees, {B{\'e}zier} extraction, polar splines and flows on surfaces}, Comput. Methods Appl. Mech. Engrg., 376 (2021), p.~113576, \url{https://doi.org/10.1016/j.cma.2020.113576}.

\bibitem{WTGC2016}
{\sc Z.~Wang, Q.~Tang, W.~Guo, and Y.~Cheng}, {\em Sparse grid discontinuous {Galerkin} methods for high-dimensional elliptic equations}, J. Comput. Phys., 314 (2016), pp.~244--263, \url{https://doi.org/10.1016/j.jcp.2016.03.005}.

\end{thebibliography}

\end{document}